\documentclass[11pt]{amsart}
\usepackage{amsfonts}
\usepackage{amsthm}
\usepackage{amsmath}
\usepackage{amscd}
\usepackage{amssymb}
\usepackage[latin2]{inputenc}
\usepackage{t1enc}
\usepackage[mathscr]{eucal}
\usepackage{indentfirst}
\usepackage{graphicx}
\usepackage{graphics}
\usepackage{pict2e}
\numberwithin{equation}{section}
\usepackage[margin=2.9cm]{geometry}
\usepackage{epstopdf}
\usepackage{mathrsfs}
\usepackage{tikz-cd}

\usepackage{placeins}
\usepackage{hyperref}

\theoremstyle{plain}
\newtheorem{Th}{Theorem}[section]

\newtheorem{``Theorem''}[Th]{``Theorem''}
\newtheorem{Lemma}[Th]{Lemma}
\newtheorem{Cor}[Th]{Corollary}

\newtheorem*{theorem**}{Theorem\theoremnum}
\newenvironment{theorem*}[1][]{%
  \edef\theoremnum{\if\relax\detokenize{#1}\relax\else~#1\fi}
  \begin{theorem**}
}{%
  \end{theorem**}
}  

\newtheorem*{corollary**}{Corollary\theoremnum}
\newenvironment{corollary*}[1][]{%
  \edef\theoremnum{\if\relax\detokenize{#1}\relax\else~#1\fi}
  \begin{corollary**}
}{%
  \end{corollary**}
}  

\theoremstyle{definition}
\newtheorem{Def}[Th]{Definition}
\newtheorem{Conj}[Th]{Conjecture}
\newtheorem{?}[Th]{Problem}

\theoremstyle{definition}
\newtheorem*{Rem}{Remark}
\newtheorem*{Rem1}{Remark 1}
\newtheorem*{Rem2}{Remark 2}
\newtheorem*{Conv}{Convention}
\newtheorem*{Que}{Question}
\newtheorem*{Ques}{Questions}
\newtheorem*{acknowledgements}{Acknowledgements}

\begin{document}

\title{Integer surgeries rational homology cobordant to lens spaces}

\author{Antony T.H. Fung}%
\address{Department of Mathematics, Boston College}
\curraddr{Department of Mathematical Sciences, Seoul National University}
\email{thf25@cantab.ac.uk}

\begin{abstract}
The Cyclic Surgery Theorem and Moser's work on surgeries on torus knots imply that for any non-trivial knot in $S^3$, there are at most two integer surgeries that produce a lens space. This paper investigates how many positive integer surgeries on a given knot in $S^3$ can produce a manifold rational homology cobordant to a lens space. Tools include Greene and McCoy's work on changemaker lattices which come from Heegaard Floer $d$-invariants, and Aceto-Celoria-Park's work on rational cobordisms and integral homology which is based on Lisca's work on lens spaces.
\end{abstract}

\maketitle

\section{Introduction}

\indent Lens spaces are one of the simplest classes of 3-manifolds, and Dehn surgery is one of the simplest ways of constructing 3-manifolds. Many classical theorems and conjectures concern when surgery on a knot in $S^3$ produces a lens space. For example, the Berge conjecture concerns ``which knots'', the Cyclic Surgery Theorem \cite{cyclic} concerns ``which slopes'', and Greene's work on the lens space realization problem \cite{realization} answers the question of ``which lens spaces''.

This paper addresses a generalization of these questions, namely when surgery on a knot in $S^3$ is smoothly rational homology cobordant to a lens space. The following definition is central to this paper:

\begin{Def} \label{lensbordant}
A \textit{lensbordant} surgery on a knot in an integer homology 3-sphere is a positive integer surgery which is smoothly rational homology cobordant to a lens space.
\end{Def}

We make the following conjecture, which is an analog of the Cyclic Surgery Theorem.

\begin{Conj} \label{conjecture}
There exists a positive integer $N$ such that for all knots $K$ in $S^3$ satisfying $\nu^+(K)\neq 0$, there are less than $N$ lensbordant surgeries on $K$.
\end{Conj}

The non-negative integer $\nu^+(K)$ is a lower bound on the 4-ball genus $g_4(K)$ coming from work by Rasmussen \cite{rasmussen} and Ni and Wu \cite{NiWu}. When $K$ is an $L$-space knot, $\nu^+(K)$ coincides with the genus $g(K)$.

We are not aware of any knot $K$ in $S^3$ satisfying $\nu^+(K)\neq 0$ having more than 5 lensbordant surgeries, though we also do not have a good reason to believe that such a knot cannot exist. There are examples of knots with 5 known lensbordant surgeries. When $K$ is the torus knot $T_{p,p+1}$, the $(p^2+p+1)$-surgery and the $(p^2+p-1)$-surgery are both lens spaces, while the $p^2$-surgery and the $(p+1)^2$-surgery both bound a smooth rational ball \cite[Th. 1.4]{aceto2017}, and hence are smoothly rational homology cobordant to $L(4,3)$. This gives 4 lensbordant surgeries. In particular, when $p=2$, the 8-surgery is smoothly rational homology cobordant to $L(2,1)$ \cite[Th. 1.3]{aceto2020surgeries}. When $p=3$, the 12-surgery produces $L(4,1)\#L(3,2)$, which is smoothly rational homology cobordant to $L(3,2)$. When $p=4$, the 20-surgery produces $L(5,1)\#L(4,3)$, which is smoothly rational homology cobordant to $L(5,1)$. Hence, when $p=2,3,4$, there are 5 known lensbordant surgeries. Note that in this paper, we take the convention where $L(p,q)$ is the $-\tfrac{p}{q}$-surgery on the unknot in $S^3$.

The first main result of this paper is

\begin{Th} \label{main result 0}
Let $K$ be a knot in $S^3$. If $\nu^+(K)\neq 0$, there is a finite number of lensbordant surgeries on $K$.
\end{Th}

This is implied by Theorem~\ref{main result 2} and Corollary~\ref{main result 5}, which are stated later in this section and proved in Section 4 and Section 5 respectively.

We do not know whether the condition $\nu^+(K)\neq 0$ can be weakened. However, we know that it is necessary to have some condition because the conclusion of the theorem is false whenever $K$ is a slice knot. When $K$ is a slice knot, all positive integer surgeries are integer homology cobordant to a lens space. We do not know whether there is any non-slice knot with an infinite number of lensbordant surgeries.

As a complement to Theorem~\ref{main result 0}, we have the following result about knots with vanishing $\nu^+$.

\begin{Th} \label{main result 3}
Let $K$ be a knot in $S^3$ with $\nu^+(K)=0$ and let $m$ be a positive integer. If the $m$-surgery on $K$ is lensbordant, then it is smoothly rational homology cobordant to $L(m,m-1)$.
\end{Th}

\begin{Rem}
When $m=1$, ``$L(m,m-1)$'' is referring to $S^3$.
\end{Rem}

Theorem~\ref{main result 3} is proved in Section 4.

Aceto, Celoria, and Park defined the notion of a reduced lens space \cite[Def. 2.2]{ACP}. Any lens space is smoothly rational homology cobordant to a reduced lens space. For simplicity, in this paper we say that $S^3$ is reduced with $p=1$, so when we say ``reduced lens space'' it also includes $S^3$.

\begin{Th} \label{main result 1}
If a positive integer surgery on a knot in $S^3$ is smoothly rational homology cobordant to a reduced lens space, then that reduced lens space must also be a positive integer surgery on a knot in $S^3$.
\end{Th}

Section 3 of this paper is dedicated to proving Theorem~\ref{main result 1}.

\begin{Ques}
Is there any interesting relationship between the two knots mentioned in Theorem~\ref{main result 1}? Is there any nice construction for knots in $S^3$ with lensbordant surgeries similar to what Berge \cite{berge} did for knots in $S^3$ with lens space surgeries?
\end{Ques}

\indent If a positive integer surgery on a knot $K$ in an integer homology 3-sphere is smoothly rational cobordant to a reduced $L(p,q)$, the surgery slope must be $r^2p$ for some positive integer $r$ \cite[Th 1.1]{ACP}. This motivates the following definition:

\begin{Def} \label{lensbordant2}
Let $r,p$ be positive integers. A surgery on a knot in an integer homology 3-sphere is \textit{$(r,p)$-lensbordant} if the surgery slope is $r^2p$ and the resulting manifold is smoothly rational homology cobordant to some reduced $L(p,q)$.
\end{Def}

Our next result is an analog of the Cyclic Surgery Theorem in the direction of Conjecture~\ref{conjecture}. The result states that for most knots $K$, when $r$ is fixed, the number of possible surgery slopes is very limited. More precisely:

\begin{Th} \label{main result 2}
Let $K$ be a knot in $S^3$ with $\nu^+(K)\neq 0$. If we fix $r$ while allowing $p$ to vary freely, the following statements hold.\\
(1) If $r=1$, there are at most 3 $(r,p)$-lensbordant surgeries on $K$. \\
(2) If $r$ is even, there are at most 2 $(r,p)$-lensbordant surgeries on $K$.\\
(3) If $r\geq 3$ is odd, there is at most 1 $(r,p)$-lensbordant surgery on $K$.
\end{Th}

Section 4 of this paper is dedicated to proving Theorem~\ref{main result 2}.

The proofs of Theorem~\ref{main result 1} and Theorem~\ref{main result 2} involve the changemaker vector $\sigma$ described in \cite{realization}. Every lens space $L(p,q)$ is the boundary of a canonical negative-definite plumbing 4-manifold $X(p,q)$ that arises from the Hirzebruch-Jung continued fraction of $\tfrac{p}{q}$. Let $n$ be the rank of $H_2(X(p,q);\mathbb{Z})$. When $L(p,q)$ is a positive integer surgery on a knot in $S^3$, the intersection form of $X(p,q)$ embeds in $-\mathbb{Z}^{n+1}$ as the orthogonal complement of some vector $\sigma$ with entries that satisfy certain combinatorial properties \cite{l-space}. We say that $\sigma$ is a changemaker associated with $L(p,q)$. Note that $L(p,q)$ can have multiple associated changemakers. In this paper, whenever we say that a lensbordant surgery is associated with a changemaker $\sigma$, we mean that the surgery is rational homology cobordant to a reduced $L(p,q)$ associated with $\sigma$. Because of Theorem~\ref{main result 1}, every lensbordant surgery is associated with some changemaker $\sigma$.

In Section 5, we establish some bounds for the surgery slope $r^2p$. At the start of Section 5, we show that:

\begin{Th} \label{main result 4}
Suppose there is an $(r,p)$-lensbordant surgery on a knot $K$ in $S^3$ with associated changemaker $\sigma$. Then
$$2\nu^+(K)+2(r-1)\geq r^2p-r|\sigma|_1\geq 2\nu^+(K).$$
\end{Th}

Using the fact that $|\sigma|_1\leq p$, this implies:

\begin{Cor} \label{main result 5}
Suppose there is an $(r,p)$-lensbordant surgery on a knot $K$ in $S^3$ with associated changemaker $\sigma$.\\
(1) If $r\geq 3$, then $\dfrac{2r^2}{(r-1)(r-2)}\nu^+(K)\geq r^2p\geq 2\nu^+(K)+r$.\\
(2) If $r\geq 2$, then $4\nu^+(K)+5\geq r^2p\geq 2\nu^+(K)+r$.
\end{Cor}

Theorem~\ref{main result 4} and Corollary~\ref{main result 5} are proved in Section 5.

For every knot $K$ in $S^3$, based on work by Rasmussen \cite{rasmussen} and Ni and Wu \cite{NiWu}, there is an associated sequence of non-negative integers $V_i(K)$. The non-negative integer $\nu^+(K)$ is the smallest index $i$ such that $V_i(K)=0$. Also, when $K$ is an $L$-space knot, the coefficients $V_i(K)$ coincide with the torsion coefficients $t_i(K)$ which can be computed from the Alexander polynomial.

\indent In Section 6, we will look in depth into the combinatorial conditions established in Section 4. We will show that:

\begin{Th}\label{imprecise theorem}
Suppose there is an $(r,p)$-lensbordant surgery on a knot $K$ in $S^3$. If $r,p$ are both even, the surgery slope $r^2p$ must be equal to $8V_0(K)$.
\end{Th}

Note that when $r,p$ are not both even, the surgery slope $r^2p$ is still very close to $8V_0(K)$. See Theorem~\ref{main result 6} in Section 6 for a precise statement.

Because of Theorem~\ref{main result 1}, we can focus on lens spaces that are positive integer surgeries on knots in $S^3$. These lens spaces are classified. Berge \cite{berge} found examples of integer surgeries on knots in $S^3$ producing lens spaces, and Rasmussen \cite{rasmussen2007} organized them nicely into a list of Berge types. Greene \cite{realization} showed that this is the complete list of lens spaces that are integer surgeries on knots in $S^3$. He showed this by analyzing the possible structures of changemakers that arise from positive integer surgeries on knots in $S^3$ producing lens spaces. Greene concluded that such a changemaker must be in one of the 33 forms that he wrote down in \cite{realization}, and then noted that for all 33 such structures, the resulting lens spaces are within this list of Berge types. Among these 33 structures in \cite{realization}, 26 are classified as ``small families'' and 7 are classified as ``large families''.

One way to tackle Conjecture~\ref{conjecture} is to show that for each of these 33 changemaker structures described by Greene, there exists some $N$ such that for all knots $K$ in $S^3$, there are at most $N$ lensbordant surgeries associated with a changemaker having that structure. In Section 7, we do this for one of the small families. This small family structure has changemakers in the form
$$\sigma=\begin{pmatrix}
    4s+3,&2s+1,&s+1,&s,&\smash{\underbrace{\begin{matrix}
        1,&\dots&,&1
    \end{matrix}}_{s\text{ copies of }1\text{'s}}}
\end{pmatrix},$$
where $s\geq 2$.

In this paper, we show that:

\begin{Th}\label{greene type 1}
Any knot in $S^3$ has at most 7 lensbordant surgeries with an associated changemaker coming from the family mentioned above.
\end{Th}

Theorem~\ref{greene type 1} is proved in Section 7.

We believe that similar techniques can be used on all 26 small families. However, we do not see how this approach extends to the 7 large families. Some large families are expected to be difficult to tackle because they contain infinite families of non-reduced lens spaces, which are currently being studied by the author in a separate project.

So far, everything we have discussed addresses surgeries on knots in $S^3$. One may ask whether this generalizes to knots in other integer homology 3-spheres. Caudell addressed the lens space realization problem for knots in the Poincar\'e homology sphere \cite{caudell} $\mathcal{P}$. In $\mathcal{P}$, we need to make more assumptions. We will show that:

\begin{Th} \label{poincare knot}
Suppose the $r^2p$-surgery on a knot $K\subset\mathcal{P}$ produces an L-space that is smoothly $\mathbb{Z}_2$-homology cobordant to a reduced $L(p,q)$. Let
$$i(K,r,p)=\begin{cases}2g(K)+r\text{ if }p\text{ is odd,} \\ 2g(K) \quad \ \ \text{ if }p\text{ is even.}\end{cases}$$
If $r^2p\geq i(K,r,p)$, then $L(p,q)$ must be a positive integer surgery on a knot in $\mathcal{P}$.\\
If $r^2p<i(K,r,p)$, then $L(p,q)$ must be a positive integer surgery on a knot in $S^3$.
\end{Th}

Theorem~\ref{poincare knot} is proved in Section 8. Note that we do not know whether the conditions in Theorem~\ref{poincare knot} can be weakened.

\begin{Que}
Does Theorem~\ref{poincare knot} generalize to knots in other integer homology spheres? One potentially interesting candidate to consider would be $\Sigma(2,3,7)$. In \cite{tange2}, Tange gave some examples of positive integer surgeries on knots in $\Sigma(2,3,7)$ yielding lens spaces.
\end{Que}

\begin{Conv}
In this paper, unless specified, homology and cohomology always take integer coefficients.
\end{Conv}

\begin{acknowledgements}
The author would like to thank his adviser Joshua Greene for inspiring this project and also for many helpful conversations throughout the project. The author would also like to thank Duncan McCoy for pointing out that using $V$ coefficients instead of torsion coefficients allows results to be stated for general knots in $S^3$ instead of just L-space knots in $S^3$, and thank Paolo Aceto for mentioning examples of knots in $S^3$ with 5 lensbordant surgeries. The research in this paper was carried out while the author was a PhD student at Boston College.
\end{acknowledgements}

\section{Analyzing the Aceto-Celoria-Park map}
In \cite{ACP}, Aceto, Celoria, and Park showed that when a closed 3-manifold $Y$ is smoothly rational homology cobordant to a reduced connected sum of lens spaces $L_Y$ (see \cite[Def. 2.2]{ACP}), there is an injective map $H_1(L_Y)\rightarrow H_1(Y)$. They proved this by showing that $H_1(L_Y)$ injects into a group that $H_1(Y)$ surjects onto, and note that if $G\rightarrow G'$ is a surjective map between finite abelian groups then there exists an injection $G'\rightarrow G$.

The goal of this section is to show that this interacts well with the $\text{spin}^\text{c}$ structures (and hence the $d$-invariants).

We use the notation in \cite{ACP}. Let $W$ be the smooth rational cobordism between $L_Y$ and $Y$. Let $P$ be the canonical negative-definite plumbed 4-manifold bounded by $L_Y$. Let $X:=P\cup_{L_Y} W$.

Throughout this section, the notation $i_*$ always refer to a homology map induced by an inclusion of topological spaces. Sometimes, we may use the same notation $i_*$ for homology maps that are induced by different inclusions, but which map each occurrence of $i_*$ is referring to should be clear from the context. In places where it is not clear, we write out the domain and codomain explicitly.

Aceto, Celoria, and Park showed that the intersection forms of $P$ and $X$ are isomorphic \cite[Prop. 4.1]{ACP}, and its proof implies that the isomorphism is induced by the inclusion $P\rightarrow X$. Since $H_2(P)$ is free, we have the decomposition
$$H_2(X)=Free(H_2(X))\oplus Torsion(H_2(X))$$
with $Free(H_2(X))$ being the image of $H_2(P)\xrightarrow{i_*} H_2(X)$.

\begin{Lemma}\label{injective}
The map $H_1(L_Y)\xrightarrow{i_*}H_1(W)$ induced by inclusion is injective.
\end{Lemma}

\begin{proof}
By excision and the neighborhood collar theorem, the inclusion $(P,L_Y)\rightarrow(X,W)$ induces isomorphisms $H_*(P,L_Y)\cong H_*(X,W)$.

Consider the following long exact sequence of the pair $(X,W)$.
$$\rightarrow H_2(W)\xrightarrow{i_*}H_2(X)\rightarrow H_2(X,W)\rightarrow$$
Since $H_1(P)=0$, $H_2(X,W)\cong H_2(P,L_Y)\cong H^2(P)$ is free, and hence the image of the map $H_2(W)\xrightarrow{i_*}H_2(X)$ contains all torsion elements of $H_2(X)$. Next, we consider the following Mayer-Vietoris sequence.
$$0\rightarrow H_2(P)\oplus H_2(W)\xrightarrow{i_*-i_*}H_2(X)\rightarrow H_1(L_Y)\xrightarrow{i_*}H_1(W)\rightarrow$$
Since the image of the map $H_2(P)\xrightarrow{i_*}H_2(X)$ is $Free(H_2(X))$ and the image of	the map $H_2(W)\xrightarrow{i_*}H_2(X)$ contains all torsion elements, we know that the map $H_2(P)\oplus H_2(W)\xrightarrow{i_*-i_*}H_2(X)$ is surjective, and hence the map $H_1(L_Y)\xrightarrow{i_*}H_1(W)$ is injective.
\end{proof}

Consider the following diagram.

\begin{tikzcd}
0\arrow[r] & H_2(X)\arrow[r,"\phi"] & H_2(X,Y)\arrow[r,"b"]\arrow[d] & H_1(Y)\arrow[r] & \ \\
& & H_2(X,W) & & \\
0\arrow[r] & H_2(P)\arrow[r,"h"]\arrow[uu,"i_*"] & H_2(P,L_Y)\arrow[r]\arrow[u,"\cong"] & H_1(L_Y)\arrow[r] & 0
\end{tikzcd}

The two rows are the long exact sequences of the pairs $(X,Y)$ and $(P,L_Y)$. The map $H_2(X,Y)\rightarrow H_2(X,W)$ comes from the long exact sequence of the triple $(Y,W,X)$. The map $H_2(P,L_Y)\rightarrow H_2(X,W)$ is excision.

For notational convenience in later parts of this section, we label the map $H_2(X)\rightarrow H_2(X,Y)$ as $\phi$, the map $H_2(P)\rightarrow H_2(P,L_Y)$ as $h$, and the map $H_2(X,Y)\rightarrow H_1(Y)$ as $b$.

\begin{Lemma}
This diagram commutes.
\end{Lemma}

\begin{proof}
Given a cycle $\alpha\in C_2(P)$, both directions map $\alpha$ to $[\alpha]\in\dfrac{C_2(X)}{C_2(W)}$, where we view $\alpha\in C_2(P)\subset C_2(X)$.
\end{proof}

As mentioned in the proof of Lemma~\ref{injective}, $H_2(X,W)$ is free. Hence, in the map $H_2(X,Y)\rightarrow H_2(X,W)$, all torsion elements of $H_2(X,Y)$ are mapped to $0$. Hence, it induces a map
$$f:\dfrac{H_2(X,Y)}{Torsion(H_2(X,Y))}\rightarrow H_2(X,W).$$
\indent Treating $Free(H_2(X))\cong Free(H_2(X))\oplus\{0\}$ as a subgroup of $H_2(X)$, we obtain the following new commutative diagram where everything except $H_1(L_Y)$ is free.

\begin{tikzcd}
0\arrow[r] & Free(H_2(X))\arrow[r,"g"] & \dfrac{H_2(X,Y)}{Torsion(H_2(X,Y))}\arrow[r]\arrow[d,"f"] & \ \\
& & H_2(X,W) & & \\
0\arrow[r] & H_2(P)\arrow[r,"h"]\arrow[uu,"\cong"] & H_2(P,L_Y)\arrow[r]\arrow[u,"\cong"] & H_1(L_Y)\arrow[r] & 0
\end{tikzcd}

The map on top, which we labeled as $g$, is obtained by composing $\phi|_{Free(H_2(X))}$ with the projection 
$$H_2(X,Y)\rightarrow\dfrac{H_2(X,Y)}{Torsion(H_2(X,Y))}.$$
Using notation in \cite[\S 4]{ACP}, let $Q_X$ be the matrix representation of the intersection form of $X$, which also represents the map $g$ under some suitable basis. Similarly, let $Q_P$ be the matrix representation of the intersection form of $P$, which also represents the map $H_2(P)\rightarrow H_2(P,L_Y)$ under some suitable basis.

Since $(H_2(P),Q_P)$ and $(H_2(X),Q_X)$ are isomorphic \cite[Prop. 4.1]{ACP}, we know that $det(Q_P)=det(Q_X)$. Since $|H_1(L_Y)|$ is finite, we know that $det(Q_P)\neq 0$. Therefore, $f$ has determinant 1, and hence we know that $f$ is an isomorphism.

By the commutativity of the diagram, we know that the image of the map $f\circ g$ coincides with the image of the map $H_2(P)\rightarrow H_2(P,L_Y)\cong H_2(X,W)$. Hence, $f$ induces an isomorphism
$$\bar{f}:\dfrac{H_2(X,Y)}{G}\rightarrow \dfrac{H_2(P,L_Y)}{h(H_2(P))}$$
\noindent where $G$ is the subgroup of $H_2(X,Y)$ generated by elements in $\phi(Free(H_2(X)))$ and $Torsion(H_2(X,Y))$.

A homomorphism cannot map a torsion element to a non-torsion element. Therefore, the image of $Torsion(H_2(X))$ under $\phi$ lies within $Torsion(H_2(X,Y))$. Hence, $G$ is exactly the subgroup of $H_2(X,Y)$ generated by elements in $\phi(H_2(X))$ and $Torsion(H_2(X,Y))$.

Now, we go back to our first diagram.

\begin{tikzcd}
0\arrow[r] & H_2(X)\arrow[r,"\phi"] & H_2(X,Y)\arrow[r,"b"]\arrow[d] & H_1(Y)\arrow[r] & \ \\
& & H_2(X,W) & & \\
0\arrow[r] & H_2(P)\arrow[r,"h"]\arrow[uu,"i_*"] & H_2(P,L_Y)\arrow[r]\arrow[u,"\cong"] & H_1(L_Y)\arrow[r] & 0
\end{tikzcd}

Let $\overline{H_1(Y)}$ be the image of $b$, which is a subgroup of $H_1(Y)$.

We see that $H_1(L_Y)\cong\dfrac{H_2(P,L_Y)}{h(H_2(P))}$ and $\dfrac{H_2(X,Y)}{\phi(H_2(X))}\cong\overline{H_1(Y)}$. Hence, we can view $\bar{f}$ as an isomorphism
$$\dfrac{\overline{H_1(Y)}}{b(Torsion(H_2(X,Y)))}\rightarrow H_1(L_Y).$$
Now we compose this with the projection
$$\overline{H_1(Y)}\rightarrow\dfrac{\overline{H_1(Y)}}{b(Torsion(H_2(X,Y)))}.$$
We obtain a surjective map $\overline{H_1(Y)}\rightarrow H_1(L_Y)$, which we denote as $s$.

Now, we consider the maps $H_1(L_Y)\xrightarrow{i_*} H_1(W)$ and $H_1(Y)\xrightarrow{i_*} H_1(W)$.

\begin{Lemma} \label{triangle}
The diagram 
\begin{tikzcd}
\overline{H_1(Y)}\arrow[rd,"i_*|_{\overline{H_1(Y)}}"]\arrow[dd,"s"] & \\
 & H_1(W) \\
H_1(L_Y)\arrow[ru,"i_*"] & 
\end{tikzcd}
 commutes.
\end{Lemma}

\begin{proof}
Let $y\in\overline{H_1(Y)}$.

Let $\alpha\in C_2(X)$ be such that $\alpha$ represents a cycle in $C_2(X,Y)$ that represents an element in $H_2(X,Y)$ that gets mapped to $y$ under $b$. By considering how the map $b$ works in the long exact sequence of a pair, we see that $y$ is represented by $\partial\alpha\in C_1(Y)\subset C_1(X)$.

By considering the chain maps being used in defining the Mayer-Vietoris sequence for $P\cup W=X$, we can write $\alpha=\beta-\gamma$, where $\beta\in C_2(P)$ and $\gamma\in C_2(W)$. Under
$$H_2(X,Y)\rightarrow H_2(X,W)\cong H_2(P,L_Y)\rightarrow H_1(L_Y),$$
the element $[\alpha]\in H_2(X,Y)$ is mapped to $[\beta]\in H_2(P,L_Y)$, which is then mapped to $[\partial\beta]\in H_1(L_Y)$, where $\partial\beta\in C_1(L_Y)\subset C_1(P)$.

Now we have $s(y)=[\partial\beta]$. From $\alpha=\beta-\gamma$, we know that $\partial\gamma=\partial\beta-\partial\alpha$. Hence, inside $C_1(W)$, $\partial\beta$ and $\partial\alpha$ differ by a boundary element, and therefore represent the same element in $H_1(W)$.
\end{proof}

\begin{Lemma}
Let $A_1,\dots,A_k$ be distinct $\text{spin}^\text{c}$ structures on $L_Y$. Then there exist distinct $\text{spin}^\text{c}$ structures $B_1,\dots,B_k$ on $Y$ such that for each $i$, $A_i$ and $B_i$ are restrictions of the same $\text{spin}^\text{c}$ structure on $W$.
\end{Lemma}

\begin{proof}
Consider the following long exact sequence.

\begin{tikzcd}
 \ \arrow[r] & H^2(W)\arrow[r,"i^*"] & H^2(L_Y)\oplus H^2(Y)\arrow[r] & H^3(W,L_Y\cup Y)\arrow[r] & \ \\
 & & H_1(L_Y)\oplus H_1(Y) \arrow[r,"i_*"]\arrow[u,leftrightarrow,"\cong"] & H_1(W)\arrow[u,leftrightarrow,"\cong"] &
\end{tikzcd}

Let $\mathfrak{s}$ be a $\text{spin}^\text{c}$ structure on $W$. For any $l$, $A_l=a_l\cdot\mathfrak{s}|_{L_Y}$ for some $a_l\in H^2(L_Y)\cong H_1(L_Y)$. By Lemma~\ref{triangle} and the surjectivity of $s$, we know that there exists some $b_l\in\overline{H_1(Y)}\subset H_1(Y)\cong H^2(Y)$ such that $(a_l,b_l)$ is in the kernel of $i_*$, and hence it is in the image of $i^*$. Let $c_l\in H^2(W)$ be such that $i^*(c_l)=(a_l,b_l)$. We set $B_l:=b_l\cdot\mathfrak{s}|_Y$. Then $c_l\cdot\mathfrak{s}$ restricts to $A_l$ and $B_l$ respectively. The $B_l$'s are distinct because given any $b_l\in\overline{H_1(Y)}$, we can use $s$ to find the corresponding $a_l$ in $H_1(L_Y)$.
\end{proof}

\begin{Cor}
The list of $d$-invariants of $Y$ is extended from the list of $d$-invariants of $L_Y$.
\end{Cor}

\begin{proof}
That is because $d$-invariants are invariant under smooth rational homology cobordism.
\end{proof}

\section{Proof of Theorem~\ref{main result 1}}

\indent This section is dedicated to proving Theorem~\ref{main result 1}.

\begin{theorem*}[\ref{main result 1}]
If a positive integer surgery on a knot in $S^3$ is smoothly rational homology cobordant to a reduced lens space, then that reduced lens space must also be a positive integer surgery on a knot in $S^3$.
\end{theorem*}

We continue the work in Section 2, using the same notation. But now $L_Y$ is a reduced $L(p,q)$, and $Y$ is the $r^2p$-surgery on $K$. We denote the $r^2p$-surgery on $K$ as $K_{r^2p}$.

The goal of this section is to show that $L(p,q)$ is a positive integer surgery on a knot in $S^3$. To show that, we will use Greene's work on the lens space realization problem \cite[Th. 1.7]{realization}. To state the statement of \cite[Th. 1.7]{realization}, we need the following definition:

\begin{Def} \cite[Def. 1.5]{realization}\label{changemaker def}
A changemaker vector is a vector $\sigma=(\sigma_0,\dots,\sigma_n)\in\mathbb{Z}^{n+1}_{>0}$ such that for all $k\in\mathbb{Z}$ satisfying $0\leq k\leq\sigma_0+\dots+\sigma_n$, there exists some $A\subseteq\{\sigma_0,\dots,\sigma_n\}$ such that the sum of the numbers in $A$ is equal to $k$.
\end{Def}

Note that we requires $\sigma$ to be in $\mathbb{Z}^{n+1}_{>0}$, which implies the entries of $\sigma$ to be all non-zero. This is consistent with \cite{realization}, but different from some other papers (for example \cite{l-space}).

In this paper, we will choose a basis in a way such that $\sigma_0\geq\dots\geq\sigma_n$, which is the opposite of the ordering used in \cite[Def. 1.5]{realization}.

Let $q'$ be such that $qq^{'}\equiv 1$ (mod $p$). Let $\Lambda(p,q)$ be the lattice coming from the intersection pairing on $P$. Recall from Section 2 that $P$ is the canonical negative-definite plumbing bounded by $L_Y$, and in this section we are setting $L_Y=L(p,q)$.

Greene showed that $\Lambda(p,q)$ embeds as the orthogonal complement to a changemaker vector in $\mathbb{Z}^{n+1}$ if and only if at least one of $L(p,q)$, $L(p,q^{'})$ appears on Berge's list of lens spaces obtained from a positive integer surgery on a knot in $S^3$ \cite[Th. 1.7]{realization}.

Since $L(p,q^{'})$ is homeomorphic to $L(p,q)$ in an orientation preserving manner, to prove Theorem~\ref{main result 1}, it suffices to show that $\Lambda(p,q)$ embeds as the orthogonal complement to a changemaker vector in $\mathbb{Z}^{n+1}$. This is what we will do in this section.

Similar to the set up in \cite[\S 2]{l-space}, we let $W_{r^2p}$ be the 4-manifold obtained by attaching a $r^2p$-framed 2-handle along $K\subset S^3=\partial D^4$. By construction, we have $H_1(W_{r^2p})=0$. The manifold $W_{r^2p}$ has boundary $K_{r^2p}$. Let $Z$ be the closed 4-manifold $X\cup_{K_{r^2p}} -W_{r^2p}$. The group $H_2(-W_{r^2p})$ is generated by the class of a surface $\Sigma$ obtained by gluing the core of the 2-handle to a Seifert surface of $K$. $\Sigma$ has the property that the intersection pairing of $[\Sigma]$ with itself gives $-r^2p$. Each $\text{spin}^\text{c}$ structure on $K_{r^2p}$ comes with a label $i\in\{0,\dots,r^2p-1\}$. This label has the property that for all $\mathfrak{s}\in\text{Spin}^\text{c}(-W_{r^2p})$ that extends the $\text{spin}^\text{c}$ structure on $K_{r^2p}$ with label $i$, we have $\langle c_1(\mathfrak{s}),[\Sigma]\rangle+r^2p\equiv 2i$ (mod $2r^2p$).

Note that we do not completely follow the notation in \cite[\S 2]{l-space}. We defined $W$ to be the cobordism in order to be consistent with \cite{ACP}, and later in this chapter we will define $\sigma$ slightly differently from \cite[\S 2]{l-space}.

\begin{figure}
\centering
\includegraphics[width=0.5\textwidth]{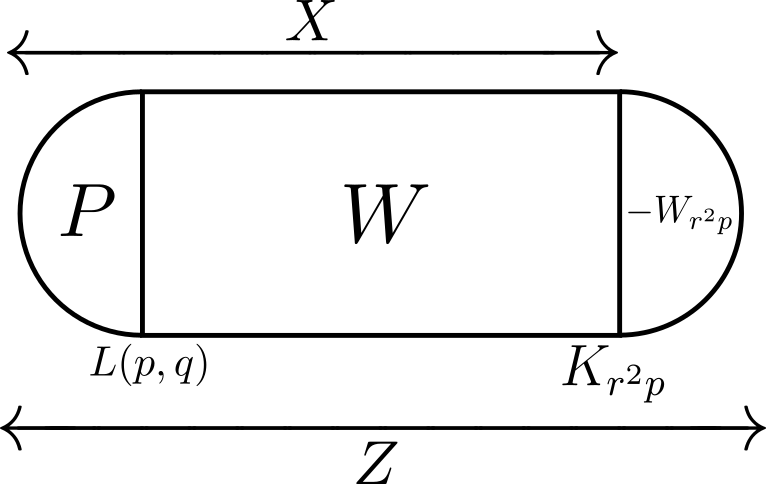}
\caption{The set up of Section 3.}
\end{figure}

\begin{Lemma} \label{primitive}
$H_2(P)$ embeds primitively in $H_2(Z)$, in the sense that for any $k\in\mathbb{Z}$ and $z\in H_2(Z)$, if $kz\in H_2(P)$, then $z$ must be equal to $p+\tau$ for some $p\in H_2(P)$ and $\tau$ a torsion element in $H_2(Z)$.
\end{Lemma}

\begin{proof}
Suppose $kz$ is in $H_2(P)$.

Consider the Mayer-Vietoris sequence
$$0\rightarrow H_2(P)\oplus H_2(W)\rightarrow H_2(P\cup W)\rightarrow H_1(L(p,q))\rightarrow H_1(P)\oplus H_1(W)\rightarrow\dots .$$
\noindent By Lemma~\ref{injective}, we know that the map $H_2(P)\oplus H_2(W)\rightarrow H_2(P\cup W)$ is an isomorphism.

Since $H_2(P)$ embeds in $H_2(P\cup W)$, we can also treat $kz$ as an element of $H_2(P\cup W)$. Now, consider the exact sequence of a pair
$$\dots\rightarrow H_3(Z,P\cup W)\rightarrow H_2(P\cup W)\rightarrow H_2(Z)\rightarrow H_2(Z,P\cup W)\rightarrow\dots .$$
By excision, we have
$$H_3(Z,P\cup W)\cong H_3(-W_{r^2p},K_{r^2p})\cong H^1(-W_{r^2p})\cong 0.$$
Also by excision, we have
$$H_2(Z,P\cup W)\cong H_2(-W_{r^2p},K_{r^2p})\cong H^2(-W_{r^2p})\cong\mathbb{Z}.$$
So, the exact sequence simplifies to
$$0\rightarrow H_2(P\cup W)\rightarrow H_2(Z)\rightarrow\mathbb{Z}\rightarrow\dots .$$
Since $kz$ is in $H_2(P\cup W)$, it is mapped to $0$ under the map $H_2(Z)\rightarrow\mathbb{Z}$. Since $\mathbb{Z}$ is free, $z$ is also mapped to $0$ under the map $H_2(Z)\rightarrow\mathbb{Z}$. Hence, $z\in H_2(P\cup W)$.

Since $H_2(P)\oplus H_2(W)\rightarrow H_2(P\cup W)$ is an isomorphism, we must have $z=p+\tau$ for some $p\in H_2(P)$ and $\tau\in H_2(W)$.
\end{proof}

\begin{Lemma} \label{ortho comp}
In the intersection form of $Z$, $H_2(P)$ is the orthogonal complement of $H_2(W\cup -W_{r^2p})$.
\end{Lemma}

\begin{proof}
From the Mayer-Vietoris sequence
$$0\rightarrow H_2(P)\oplus H_2(W\cup -W_{r^2p})\rightarrow H_2(Z)\rightarrow H_1(L(p,q))\rightarrow\dots ,$$
we know that the map $H_2(P)\oplus H_2(W\cup -W_{r^2p})\rightarrow H_2(Z)$ is a full rank embedding. Since $H_2(P)$ and $H_2(W\cup -W_{r^2p})$ are orthogonal to each other, Lemma~\ref{primitive} implies that $H_2(P)$ is the orthogonal complement of $H_2(W\cup -W_{r^2p})$.
\end{proof}

\begin{Lemma} \label{sur tor}
The image of the map $H_2(W)\rightarrow H_2(W\cup -W_{r^2p})$ induced by inclusion is exactly the torsion subgroup of $H_2(W\cup -W_{r^2p})$.
\end{Lemma}

\begin{proof}
By excision and the neighborhood collar theorem, we have $H_2(W\cup -W_{r^2p},W)\cong H_2(-W_{r^2p},K_{r^2p})\cong H^2(-W_{r^2p})$, which has no torsion because $H_1(-W_{r^2p})=0$.

Consider the exact sequence of a pair
$$\rightarrow H_2(W)\rightarrow H_2(W\cup -W_{r^2p})\rightarrow H_2(W\cup -W_{r^2p},W)\rightarrow.$$
Since $H_2(W\cup -W_{r^2p},W)$ is free, the image of the map $H_2(W)\rightarrow H_2(W\cup -W_{r^2p})$ must contain all torsion elements. Since $H_2(W)$ is finite, the image is exactly the torsion subgroup.
\end{proof}

\begin{Lemma} \label{size rp}
The map $H_1(L(p,q))\oplus H_1(K_{r^2p})\rightarrow H_1(W)$ induced by the inclusion $L(p,q)\cup K_{r^2p}\rightarrow W$ has kernel of size $rp$ and image of size $rp$.
\end{Lemma}

\begin{proof}
Consider the exact sequence of a pair

$$
\begin{aligned}
&0\rightarrow H_2(W)\rightarrow H_2(W,L(p,q)\cup K_{r^2p})\rightarrow H_1(L(p,q)\cup K_{r^2p})\rightarrow H_1(W)\\
&\rightarrow H_1(W,L(p,q)\cup K_{r^2p})\rightarrow\mathbb{Z}^2\rightarrow\dots .
\end{aligned}
$$

Using Poincar\'{e}-Lefschetz duality and universal coefficient theorem, we get

$$
\begin{aligned}
&0\rightarrow H_2(W)\rightarrow H_1(W)\rightarrow H_1(L(p,q))\oplus H_1(K_{r^2p})\rightarrow H_1(W)\\
&\rightarrow H_2(W)\oplus\mathbb{Z}\rightarrow\mathbb{Z}^2\rightarrow\dots ,
\end{aligned}
$$

which simplifies to
$$0\rightarrow H_2(W)\rightarrow H_1(W)\rightarrow H_1(L(p,q))\oplus H_1(K_{r^2p})\rightarrow H_1(W)\rightarrow H_2(W)\rightarrow 0.$$
Hence, the map $H_1(L(p,q))\oplus H_1(K_{r^2p})\rightarrow H_1(W)$ has kernel of size $\dfrac{|H_1(W)|}{|H_2(W)|}$ and image of size $\dfrac{|H_1(W)|}{|H_2(W)|}$. Since the kernel and the image has the same size, their sizes must both be $\sqrt{|H_1(L(p,q))\oplus H_1(K_{r^2p})|}=rp$.
\end{proof}

\begin{Lemma} \label{Zr}
The kernel of the map $H_1(K_{r^2p})\rightarrow H_1(W)$ induced by inclusion is isomorphic to $\mathbb{Z}_r$ (cyclic group of order $r$).
\end{Lemma}

\begin{proof}
Due to the surjectivitiy of the map labeled as ``$s$'' in Lemma~\ref{triangle}, we know that the image of the map $H_1(L(p,q))\rightarrow H_1(W)$ is contained in the image of the map $H_1(K_{r^2p})\rightarrow H_1(W)$. Hence, the image of the map $H_1(L(p,q))\oplus H_1(K_{r^2p})\rightarrow H_1(W)$ is the same as the image of the map $H_1(K_{r^2p})\rightarrow H_1(W)$.

Therefore, by Lemma~\ref{size rp}, the map $H_1(K_{r^2p})\rightarrow H_1(W)$ has an image of size $rp$ and a kernel of size $\dfrac{r^2p}{rp}=r$. Since $H_1(K_{r^2p})$ is cyclic, the kernel is a cyclic group of order $r$.
\end{proof}

Recall that $H_2(-W_{r^2p})$ is generated by $[\Sigma]$. The following Lemma is about the generator of the free part of $H_2(W\cup -W_{r^2p})$:

\begin{Lemma} \label{mult r}
The map $H_2(-W_{r^2p})\rightarrow H_2(W\cup -W_{r^2p})$ induced by inclusion maps $[\Sigma]$ to $r$ times a generator of the free part of $H_2(W\cup -W_{r^2p})$ plus a torsion element (can be $0$).
\end{Lemma}

\begin{proof}
Consider the Mayer-Vietoris sequence
$$0\rightarrow H_2(W)\oplus H_2(-W_{r^2p})\rightarrow H_2(W\cup -W_{r^2p})\rightarrow H_1(K_{r^2p})\rightarrow H_1(W)\rightarrow\dots .$$
By Lemma~\ref{Zr}, we can simplify this into
$$0\rightarrow H_2(W)\oplus H_2(-W_{r^2p})\rightarrow H_2(W\cup -W_{r^2p})\rightarrow\mathbb{Z}_r\rightarrow 0.$$
Combining this with Lemma~\ref{sur tor}, we conclude that $H_2(-W_{r^2p})\rightarrow H_2(W\cup -W_{r^2p})$ has to map $[\Sigma]$ to $r$ times a generator of the free part of $H_2(W\cup -W_{r^2p})$ plus a torsion element (can be $0$).
\end{proof}

Let $\sigma\in H_2(W\cup -W_{r^2p})$ be a generator of the free part such that $[\Sigma]=r\sigma+torsion$.

Recall that the intersection pairing of $[\Sigma]$ with itself gives $-r^2p$. So, we have
$$\langle\sigma,\sigma\rangle=\dfrac{1}{r^2}\langle[\Sigma],[\Sigma]\rangle=\dfrac{-r^2p}{r^2}=-p.$$
Since $P$ is negative-definite, by Lemma~\ref{ortho comp}, we know that $Z$ is negative-definite. By Donaldson's Theorem \cite{donaldson}, we know that the intersection pairing on $Z$ is $-\mathbb{Z}^{n+1}$, where $n$ is the rank of $H_2(P)$. Moreover, under the identification $H_2(Z)\cong-\mathbb{Z}^{n+1}$, the image of the first Chern class map $c_1:\text{Spin}^\text{c}(Z)\rightarrow H^2(Z)\cong H_2(Z)\cong-\mathbb{Z}^{n+1}$ is exactly the set of characteristic vectors
$$\text{Char}(-\mathbb{Z}^{n+1}):=\{v\in-\mathbb{Z}^{n+1}|\forall w\in-\mathbb{Z}^{n+1} \ \langle v,w\rangle\equiv\langle w,w\rangle\text{ (mod 2)}\},$$
where $H^2(Z)\cong H_2(Z)$ is Poincar\'e duality.

In any orthonormal basis, the set of characteristic vectors of $-\mathbb{Z}^{n+1}$ is exactly the set of vectors with all entries being odd.

Now, we can prove the following:

\begin{Lemma} \label{mult r spinc}
A $\text{spin}^\text{c}$ structure on $K_{r^2p}$ extends to $W$ if and only if the label $i$ on the $\text{spin}^\text{c}$ structure satisfies $i\equiv\begin{cases}\dfrac{r}{2}\quad(mod \ r)\text{   if }r\text{ is even and }p\text{ is odd,}\\0\quad(mod \ r)\text{   otherwise.}\end{cases}$.
\end{Lemma}

Note that since $P$ and $-W_{r^2p}$ have trivial $H_1$, a $\text{spin}^\text{c}$ structure on $K_{r^2p}$ extends to $W$ if and only if it extends to $Z$. So, in the statement of Lemma~\ref{mult r spinc}, we can always replace the $W$ with $Z$.

\begin{proof}[Proof of Lemma~\ref{mult r spinc}]
Consider the map $H_1(L(p,q))\oplus H_1(K_{r^2p})\rightarrow H_1(W)$ mentioned in Lemma~\ref{size rp}. By Lemma~\ref{size rp}, it has a kernel of size $rp$. Since the map $H_1(L(p,q))\rightarrow H_1(W)$ is injective, the kernel of $H_1(L(p,q))\oplus H_1(K_{r^2p})\rightarrow H_1(W)$ cannot contain two different elements with the same $H_1(K_{r^2p})$ component. Hence, the $rp$ elements in the kernel of $H_1(L(p,q))\oplus H_1(K_{r^2p})\rightarrow H_1(W)$ all have different $H_1(K_{r^2p})$ components.

Therefore, by applying Poincar\'{e}-Lefschetz duality on the exact sequence of a pair
$$\dots\rightarrow H_2(W,L(p,q)\cup K_{r^2p})\rightarrow H_1(L(p,q))\oplus H_1(K_{r^2p})\rightarrow H_1(W)\rightarrow\dots ,$$
we can conclude that the map $H^2(W)\rightarrow H^2(K_{r^2p})$ has image of size $rp$. Therefore, there are exactly $rp$ $\text{spin}^\text{c}$ structures on $K_{r^2p}$ that extend to $W$.

Notice that there are exactly $rp$ possible values of $i \ (mod \ r^2p)$ satisfying
\begin{equation}\label{eq:1}
i\equiv\begin{cases}\dfrac{r}{2}\quad(mod \ r)\text{   if }r\text{ is even and }p\text{ is odd,}\\0\quad(mod \ r)\text{   otherwise.}\end{cases}
\end{equation}
Therefore, to prove Lemma~\ref{mult r spinc}, it suffices to show that a $\text{spin}^\text{c}$ structure on $K_{r^2p}$ extends to $W$ implies that its label $i$ satisfies (\ref{eq:1}). By the definition of the labeling, this is equivalent to showing that for any $\text{spin}^\text{c}$ structure $\mathfrak{s}$ on $Z$,
$$\langle c_1(\mathfrak{s}),[\Sigma]\rangle+r^2p\equiv\begin{cases}r\quad(mod \ 2r)\text{   if }r\text{ is even and }p\text{ is odd,}\\0\quad(mod \ 2r)\text{   otherwise.}\end{cases}$$
\noindent Since $\langle c_1(\mathfrak{s}),[\Sigma]\rangle+r^2p=r(\langle c_1(\mathfrak{s}),\sigma\rangle+rp)$, we only need to figure out whether $\langle c_1(\mathfrak{s}),\sigma\rangle+rp$ is odd or even.

Since $c_1(\mathfrak{s})$ is a characteristic vector, we have
$$\langle c_1(\mathfrak{s}),\sigma\rangle\equiv\langle \sigma,\sigma\rangle=-p\quad(mod \ 2).$$
Hence, $\langle c_1(\mathfrak{s}),\sigma\rangle+rp$ is odd when $r$ is even and $p$ is odd, and it is even otherwise.
\end{proof}

By \cite[Th. 1.7]{realization} and Lemma~\ref{ortho comp}, to prove Theorem~\ref{main result 1}, we only need to show that $\sigma$ embeds as a changemaker vector in the lattice $-\mathbb{Z}^{n+1}$.

By choosing suitable orthonormal basis, we let $\sigma=(\sigma_0,\dots,\sigma_n)$ with
$$0\leq\sigma_n\leq\dots\leq\sigma_0.$$
For every knot $K$ in $S^3$, based on work by Rasmussen \cite{rasmussen} and also Ni and Wu \cite{NiWu}, there is a sequence of non-negative integers $V_i(K)$ satisfying the following properties:
\begin{itemize}
\item $\forall i\geq 0$, $V_i(K)-V_{i+1}(K)\in\{0,1\}.$
\item $V_{g(K)}(K)=0$, where $g(K)$ is the genus of the knot.
\item For all positive integer $k$, for all $i$ with $0\leq i\leq k-1$, we have
$$-2V_{\text{min}\{i,k-i\}}=d(K_k,i)-d(U_k,i),$$
where $K_k$ is the $k$-surgery on $K$, and $U_k$ is the $k$-surgery on the unknot (which is $L(k,k-1)$). The $i$ behind denotes the $\text{spin}^\text{c}$ structure with label $i$, and the $d$ is the $d$-invariant.
\end{itemize}

The last property implies that for $0\leq i\leq\dfrac{r^2p}{2}$, we have
$$-2V_i(K)=d(K_{r^2p},i)-d(U_{r^2p},i).$$
The smooth concordance invariant $\nu^+(K)$ is the smallest $i$ such that $V_i(K)=0$. It is a lower bound of the 4-ball genus of $K$ \cite[Prop. 2.4]{HomWu}. When $K$ is an L-space knot, $\nu^+$ coincides with the genus of $K$ and these $V$ coefficients coincide with the torsion coefficients that can be computed from the Alexander polynomial.

For notational simplicity, we often drop the $(K)$ and only write $V_i$ and $\nu^+$.

We now proceed to prove that $\sigma$ is a changemaker vector. Our proof will follow the argument in \cite[\S2-3]{l-space} that shows that $\sigma$ is a changemaker. The only difference is that we use the $V$ coefficients instead of torsion coefficients because we are not assuming $K$ to be an L-space knot, and there are some other minor changes because of the $r$ factor coming from our cobordism which did not exist in \cite[\S2-3]{l-space}.

Based on Lemma~\ref{mult r spinc}, we make the following definition.

\begin{Def} \label{relevant}
We say that the coefficient $V_i$ is a \textit{relevant coefficient} if the index satisfies $0\leq i\leq\dfrac{r^2p}{2}$ and
$$i\equiv\begin{cases}\dfrac{r}{2}\quad(mod \ r)\text{   if }r\text{ is even and }p\text{ is odd,}\\0\quad(mod \ r)\text{   otherwise.}\end{cases}$$
\noindent In that case, we say that the index $i$ is \textit{relevant}.
\end{Def}

\begin{Lemma} \label{8V}
Let $i$ be relevant. Then
$$-8V_i=\max_{\substack{\mathfrak{c}\in Char(-\mathbb{Z}^{n+1}) \\ \langle\mathfrak{c},[\Sigma]\rangle+r^2p=2i}}(\mathfrak{c}^2+(n+1)).$$
\end{Lemma}

\begin{proof}[Proof of Lemma~\ref{8V}]
Let $i$ be relevant. There is a $\text{spin}^\text{c}$ structure on $K_{r^2p}$ with label $i$. By Lemma~\ref{mult r spinc}, this $\text{spin}^\text{c}$ structure extends to some $\text{spin}^\text{c}$ structure $\mathfrak{s}^W_i$ on $W$. Let $\mathfrak{s}^{L(p,q)}_i$ be the restriction of $\mathfrak{s}^W_i$ on $L(p,q)$. (Warning: In the literature, there is a convention of labeling $\text{spin}^\text{c}$ structures on lens spaces with numbers $0,\dots,p-1$. The label of $\mathfrak{s}^{L(p,q)}_i$ under such a convention may not necessarily be $i$.)

Now, we recreate \cite[Lemma 2.3]{l-space}, but under the context of this paper. For any $\text{spin}^\text{c}$ structure $\mathfrak{s}^P_i$ on $P$ that extends $\mathfrak{s}^{L(p,q)}_i$, we have
$$c_1(\mathfrak{s}^P_i)^2+n\leq 4d(\mathfrak{s}^{L(p,q)}_i).$$
Due to the sharpness of $P$ (see the discussion between Prop. 2.2 and Lemma 2.3 of \cite{l-space}), $\mathfrak{s}^P_i$ can be chosen such that the equality holds.

Similarly, for any $\text{spin}^\text{c}$ structure $\mathfrak{s}^{-W_{r^2p}}_i$ on $-W_{r^2p}$ that extends $(K_{r^2p},i)$, the argument in the proof of \cite[Lemma 2.3]{l-space} implies that
$$c_1(\mathfrak{s}^{-W_{r^2p}}_i)^2+1\leq -4d(U_{r^2p},i),$$
where $\mathfrak{s}^{-W_{r^2p}}_i$ can be chosen such that equality holds.

Adding these two inequalities together, we have
$$c_1(\mathfrak{s}^P_i)^2+c_1(\mathfrak{s}^{-W_{r^2p}}_i)^2+(n+1)\leq 4d(\mathfrak{s}^{L(p,q)}_i)-4d(U_{r^2p},i),$$
where $\mathfrak{s}^P_i$ and $\mathfrak{s}^{-W_{r^2p}}_i$ can be chosen such that equality holds.

Now, we look at the cobordism $W$ that exists in this paper but not in \cite{l-space}. Since the $d$-invariant is invariant under smooth rational homology cobordism, we know that
$$d(\mathfrak{s}^{L(p,q)}_i)=d(K_{r^2p},i).$$
Also, since $L(p,q)$ and $K_{r^2p}$ are rational homology spheres and $H_2(W)$ is finite, we have
$$c_1(\mathfrak{s})^2=c_1(\mathfrak{s}^P_i)^2+c_1(\mathfrak{s}^{-W_{r^2p}}_i)^2,$$
where $\mathfrak{s}$ is the $\text{spin}^\text{c}$ structure obtained by gluing the $\text{spin}^\text{c}$ structures $\mathfrak{s}^P_i$, $\mathfrak{s}^W_i$, and $\mathfrak{s}^{-W_{r^2p}}_i$ together.

Hence, for any $\text{spin}^\text{c}$ structure $\mathfrak{s}$ on $Z$ that extends $(K_{r^2p},i)$, we have
$$c_1(\mathfrak{s})^2+(n+1)\leq 4d(K_{r^2p},i)-4d(U_{r^2p},i)=-8V_i.$$
And for any relevant $i$, $\mathfrak{s}$ can be chosen such that equality holds.

Due to the way the labeling of $\text{spin}^\text{c}$ structures on $K_{r^2p}$ works, we know that the $\text{spin}^\text{c}$ structures $\mathfrak{s}$ on $Z$ that extend $(K_{r^2p},i)$ are exactly the ones that satisfy
\begin{equation}\label{eq:2}
\langle c_1(\mathfrak{s}),[\Sigma]\rangle+r^2p\equiv 2i\text{ (mod }2r^2p\text{).}
\end{equation}
Hence, we conclude that
\begin{equation}\label{eq:3}
-8V_i=\max_{\substack{\mathfrak{c}\in Char(-\mathbb{Z}^{n+1}) \\ \langle\mathfrak{c},[\Sigma]\rangle+r^2p\equiv 2i\text{ (mod }2r^2p\text{)}}}(\mathfrak{c}^2+(n+1)).
\end{equation}
The equation \indent (\ref{eq:3}) is very close to the statement of Lemma~\ref{8V}. The only difference is that in Lemma~\ref{8V}, it says ``$=2i$'' instead of ``$\equiv 2i$ (mod $2r^2p$)'' under the $\max$.

To finish the proof of Lemma~\ref{8V}, it suffices to show that for any $\mathfrak{c}=c_1(\mathfrak{s})\in Char(-\mathbb{Z}^{n+1})$ satisfying (\ref{eq:2}), there exists some $\mathfrak{c}'$ with ${\mathfrak{c}'}^2\geq\mathfrak{c}^2$ satisfying
$$\langle \mathfrak{c}',[\Sigma]\rangle+r^2p=2i.$$
To do this, we apply the argument in the proof of \cite[Lemma 2.9]{mccoy}. Let $\mathfrak{c}=c_1(\mathfrak{s})\in Char(-\mathbb{Z}^{n+1})$ satisfying (\ref{eq:2}). Then
$$\langle\mathfrak{c},[\Sigma]\rangle+r^2p=2i+2mr^2p$$
for some $m\in\mathbb{Z}$.

Let $\mathfrak{c}':=\mathfrak{c}+2m[\Sigma]$. Then we have
$$\langle\mathfrak{c}',[\Sigma]\rangle+r^2p=2i+2mr^2p+2m\langle[\Sigma],[\Sigma]\rangle=2i,$$
and also 
\begin{align*}
{\mathfrak{c}'}^2&=\mathfrak{c}^2+4m\langle\mathfrak{c},[\Sigma]\rangle+4m^2\langle[\Sigma],[\Sigma]\rangle\\
&=\mathfrak{c}^2+4m(2i+2mr^2p-r^2p)+4m^2(-r^2p)\\
&=\mathfrak{c}^2+4r^2pm^2+(8i-4r^2p)m.
\end{align*}
Viewing $4r^2pm^2+(8i-4r^2p)m$ as a quadratic in $m$, the roots are 0 and $(1-\dfrac{2i}{r^2p})$. Since $i$ is relevant, we have $0\leq 2i\leq r^2p$. Hence, there is no integer strictly between $0$ and $(1-\dfrac{2i}{r^2p})$. Therefore, since $m$ is an integer, we have $4r^2pm^2+(8i-4r^2p)m\geq 0$, and hence ${\mathfrak{c}'}^2\geq\mathfrak{c}^2$.
\end{proof}

\begin{Lemma} \label{is changemaker}
The vector $\sigma$ is a changemaker vector.
\end{Lemma}

\begin{proof}
First of all, similar to the argument in \cite[\S 3.4]{realization}, by Lemma 3.3 and the indecomposability of linear lattices \cite[Cor. 3.5]{realization}, we know that all entries of $\sigma$ are non-zero.

When $\mathfrak{c}=(1,\dots,1)$, we have
\begin{align*}
\langle\mathfrak{c},[\Sigma]\rangle+r^2p &= r\langle\mathfrak{c},\sigma\rangle+r^2p\\
&=-r|\sigma|_1+r^2p\\
&=2\left(\dfrac{r^2p-r|\sigma|_1}{2}\right).
\end{align*}
Also, we have $\mathfrak{c}^2+(n+1)=0$.

By Lemma~\ref{8V}, since those $V_i$ coefficients are all non-negative, we can conclude that
$$V_{\tfrac{1}{2}r(rp-|\sigma|_1)}=0.$$
Therefore, $\nu^+\leq\tfrac{1}{2}r(rp-|\sigma|_1)$. Hence, for all $i\geq\tfrac{1}{2}r(rp-|\sigma|_1)$, we have $V_i=0$.

By Lemma~\ref{8V}, we know that for all relevant $i\geq\tfrac{1}{2}r(rp-|\sigma|_1)$, there exists some $\mathfrak{c}\in Char(-\mathbb{Z}^{n+1})$ such that $\mathfrak{c}^2=-(n+1)$ and $\langle\mathfrak{c},[\Sigma]\rangle+r^2p=2i$. This is equivalent to saying that there exists some $\mathfrak{c}\in\{\pm1\}^{n+1}$ such that
$$r\langle\mathfrak{c},\sigma\rangle+r^2p=2i.$$
For the rest of the proof, we will separate the cases according to Lemma~\ref{mult r spinc}. The proofs for the two cases are almost identical.

\underline{Case 1: $r$ is odd or $p$ is even.}\\
Let $i$ be relevant and $i\geq\tfrac{1}{2}r(rp-|\sigma|_1)$. When $r$ is odd or $p$ is even, that is equivalent to saying that $i$ is a multiple of $r$ that satisfies
$$\tfrac{1}{2}r(rp-|\sigma|_1)\leq i\leq\tfrac{r^2p}{2}.$$
Let $i':=\dfrac{i}{r}$. Then we have $\tfrac{rp-|\sigma'|_1}{2}\leq i'\leq\tfrac{rp}{2}$, and there exists $\mathfrak{c}\in\{\pm1\}^{n+1}$ such that
$$r\langle \mathfrak{c},\sigma\rangle+r^2p=2i'r,$$
which simplifies to
\begin{equation}\label{eq:4}
\langle \mathfrak{c},\sigma\rangle+rp=2i'.
\end{equation}
Let $i'':=rp-i'$. Then we have $\tfrac{rp}{2}\leq i''\leq\tfrac{rp+|\sigma|_1}{2}$. Substituting $i'=rp-i''$ into (\ref{eq:4}) and rearranging, we get
$$\langle -\mathfrak{c},\sigma\rangle+rp=2i''.$$
Notice that $\mathfrak{c}\in\{\pm1\}^{n+1}$ implies $-\mathfrak{c}\in\{\pm1\}^{n+1}$. Hence, we can conclude that for any integer $i'$ satisfying $\tfrac{rp-|\sigma'|_1}{2}\leq i'\leq\tfrac{rp+|\sigma'|_1}{2}$, there exists $\mathfrak{c}\in\{\pm1\}^{n+1}$ such that (\ref{eq:4}) holds.

Note that since $|\sigma|_1\equiv\langle\sigma,\sigma\rangle=-p$ (mod $2$), $\tfrac{rp\pm|\sigma|_1}{2}$ must be integers under Case 1.

Setting $j:=i'-\tfrac{rp+|\sigma|_1}{2}$ and substituting it into (\ref{eq:4}), we arrive to the conclusion that for any integer $j$ satisfying $-|\sigma|_1\leq j\leq0$, there exists $\mathfrak{c}\in\{\pm1\}^{n+1}$ such that
$$\langle \mathfrak{c},\sigma\rangle=2j+|\sigma|_1.$$
Setting $\mathfrak{c}=(-1,\dots,-1)+2\chi$ for some $\chi\in\{0,1\}^{n+1}$, we arrive to the conclusion that for any integer $j$ satisfying $-|\sigma|_1\leq j\leq0$, there exists $\chi\in\{0,1\}^{n+1}$ such that
$$\langle \chi,\sigma\rangle=j.$$
Together with the fact that all entries of $\sigma$ are non-zero, this implies that $\sigma$ is a changemaker vector.

\underline{Case 2: $r$ is even and $p$ is odd.}\\
Let $i$ be relevant and $i\geq\tfrac{1}{2}r(rp-|\sigma|_1)$. When $r$ is even and $p$ is odd, that is equivalent to saying that $i$ satisfies $i\equiv\tfrac{r}{2}$ (mod $r$) and
$$\tfrac{1}{2}r(rp-|\sigma|_1)\leq i\leq\tfrac{r^2p}{2}.$$
Let $i':=\tfrac{i}{r}-\tfrac{1}{2}$. Then we have $\tfrac{rp-|\sigma'|_1}{2}\leq i'+\tfrac{1}{2}\leq\tfrac{rp}{2}$, and there exists $\mathfrak{c}\in\{\pm1\}^{n+1}$ such that
$$r\langle\mathfrak{c},\sigma\rangle+r^2p=2\left(i'+\tfrac{1}{2}\right)r,$$
which simplifies to
\begin{equation}\label{eq:5}
\langle\mathfrak{c},\sigma\rangle+rp=2i'+1.
\end{equation}
Let $i'':=rp-i'-1$. Then $\tfrac{rp}{2}\leq i''+\tfrac{1}{2}\leq\tfrac{rp+|\sigma|_1}{2}$ and
$$\langle -\mathfrak{c},\sigma\rangle+rp=2i''+1.$$
Hence, (\ref{eq:5}) holds for $\tfrac{rp-|\sigma|_1}{2}\leq i''+\tfrac{1}{2}\leq\tfrac{rp+|\sigma|_1}{2}$.

Note that in Case 2, $rp+|\sigma|_1$ is odd. Therefore, $\tfrac{rp+|\sigma|_1-1}{2}$ is an integer.

Setting $j:=i'-\tfrac{rp+|\sigma'|_1-1}{2}$ and substituting it into (\ref{eq:5}), we see that for any integer $j$ satisfying $-|\sigma'|_1\leq j\leq0$, there exists $\mathfrak{c}\in\{\pm1\}^{n+1}$ such that
$$\langle \mathfrak{c},\sigma\rangle=2j+|\sigma|_1.$$
The rest of the proof is the same as Case 1.
\end{proof}

\begin{proof}[Proof of Theorem~\ref{main result 1}]
Theorem~\ref{main result 1} follows from \cite[Th. 1.7]{realization}, Lemma~\ref{ortho comp}, and Lemma~\ref{is changemaker}.
\end{proof}

\section{Number of possible surgeries when fixing \texorpdfstring{$r$}{r}}

In this section, we will show that when we fix $K\subset S^3$ and a value of $r$, while allowing $p$ to vary freely, the number of $(r,p)$-lensbordant surgeries is very limited. We will follow an argument similar to \cite[\S 2]{mccoy}, with some changes to take $r$ into account.

Recall that we defined the coefficient $V_i$ to relevant if $0\leq i\leq\dfrac{r^2p}{2}$ and
$$i\equiv\begin{cases}\dfrac{r}{2}\quad(mod \ r)\text{   if }r\text{ is even and }p\text{ is odd,}\\0\quad(mod \ r)\text{   otherwise.}\end{cases}$$
Sometimes it is convenient to consider the relevant coefficients as a sequence itself, as defined below.

\begin{Def}\label{rel coeff as a seq}
For $i$ satisfying
$$0\leq i\leq\begin{cases}\frac{rp}{2}-1\text{   if }r\text{ is even and }p\text{ is odd,}\\ \frac{rp-1}{2}\text{   if }r,p\text{ are both odd,}\\ \frac{rp}{2}\text{   if }p\text{ is even,}\end{cases}$$
\noindent we define
$$V^{rel}_i:=\begin{cases}V_{\frac{r}{2}+ir}\text{   if }r\text{ is even and }p\text{ is odd.}\\V_{ir}\text{   otherwise.}\end{cases}$$
\end{Def}

This range of $i$ for $V^{rel}_i$ corresponds to the range of $i$ for $V_i$ to be relevant.

Since the $V$ coefficients of a knot are always monotonic decreasing, the relevant coefficients are also monotonic decreasing. However, consecutive relevant coefficients can differ by more than 1.

We define $\nu^{+rel}$ to be the number of non-zero relevant coefficients. Equivalently,
$$\nu^{+rel}:=min\{i|V^{rel}_i=0\}.$$
\noindent Similar to how McCoy \cite{mccoy} defined $T_m$, we define $T^{rel}_m$ to be the number of relevant coefficients less than or equal to $m$. Equivalently,
$$T^{rel}_m:=|\{0\leq i<\nu^{+rel} \mid 0<V^{rel}_i\leq m\}|.$$
Note that due to monotonicity, we have $\nu^{+rel}=T^{rel}_{V^{rel}_0}$.

We define $S_m=\{\alpha\in\mathbb{Z}^{dim(\sigma)}_{\geq 0} \mid \sum\limits_j\alpha_j(\alpha_j+1)=2m\}$. This is exactly the same as how McCoy defined $S_m$ in \cite{mccoy}.

We also define $S_{\leq m}=\{\alpha\in\mathbb{Z}^{dim(\sigma)}_{\geq 0} \mid \sum\limits_j\alpha_j(\alpha_j+1)\leq 2m\}$. Due to the fact that consecutive relevant coefficients can decrease by more than 1, some arguments later in this section will not work if we use $S_m$ like how it is used in \cite{mccoy}. Instead, we will need to consider $S_{\leq m}$.

First we prove the following properties, which is an analog of \cite[Lemma 2.10]{mccoy}.

\begin{Lemma} \label{linear system}
The following three statements are true.\\
(1) For all $m$ satisfying $0\leq m<V^{rel}_0$, we have $T^{rel}_m=\displaystyle\max_{\alpha\in S_{\leq m}}\sigma\cdot\alpha$.\\
\\
(2) $\nu^{+rel}=\begin{cases}-\dfrac{1}{2}+\dfrac{1}{2}(rp-|\sigma|_1)\text{   if }r\text{ is even and }p\text{ is odd,}\\ \dfrac{1}{2}(rp-|\sigma|_1)\text{   otherwise.}\end{cases}$\\
\\
(3) $\nu^{+rel}=T^{rel}_{V^{rel}_0}\leq\displaystyle\max_{\alpha\in S_{V^{rel}_0}}\sigma\cdot\alpha$.
\end{Lemma}

Note that when $r\geq 2$, statement (1) holds for $m=V^{rel}_0$ as well. We will prove this in Section 6 (Corollary~\ref{linear system equality}) after we study more about the behavior of relevant coefficients.

\begin{proof}[Proof of Lemma~\ref{linear system}]
Lemma~\ref{8V} states when for any relevant $i$, 
$$-8V_i=\max_{\substack{\mathfrak{c}\in Char(-\mathbb{Z}^{n+1}) \\ \langle\mathfrak{c},[\Sigma]\rangle+r^2p=2i}}(\mathfrak{c}^2+(n+1)),$$
\noindent which can be rewritten as
\begin{align*}
-8V^{rel}_i&=\max_{\substack{\mathfrak{c}\in Char(-\mathbb{Z}^{n+1}) \\ r\langle\mathfrak{c},\sigma\rangle+r^2p=2ir+r}}(\mathfrak{c}^2+(n+1))\text{   if }r\text{ is even and }p\text{ is odd,}\\
-8V^{rel}_i&=\max_{\substack{\mathfrak{c}\in Char(-\mathbb{Z}^{n+1}) \\ r\langle\mathfrak{c},\sigma\rangle+r^2p=2ir}}(\mathfrak{c}^2+(n+1))\text{   otherwise.}
\end{align*}
\noindent The condition of $\mathfrak{c}$ being a characteristic vector can be stated as $\mathfrak{c}=2\alpha+(1,\dots,1)$ for some $\alpha\in\mathbb{Z}^{n+1}$. Now we can rewrite the above as
\begin{align*}
-8V^{rel}_i&=\max_{rp-|\sigma|_1-2\sum\limits_j\sigma_j\alpha_j=2i+1}((n+1)-\sum\limits_j((2\alpha_j+1)^2))\text{   if }r\text{ is even and }p\text{ is odd,}\\
-8V^{rel}_i&=\max_{rp-|\sigma|_1-2\sum\limits_j\sigma_j\alpha_j=2i}((n+1)-\sum\limits_j((2\alpha_j+1)^2))\text{   otherwise.}
\end{align*}
\noindent This can be further simplified into
\begin{align}
2V^{rel}_i&=\min_{rp-|\sigma|_1-2\sum\limits_j\sigma_j\alpha_j=2i+1}\sum\limits_j\alpha_j(\alpha_j+1)\text{   if }r\text{ is even and }p\text{ is odd,}\label{eq:6}\\
2V^{rel}_i&=\min_{rp-|\sigma|_1-2\sum\limits_j\sigma_j\alpha_j=2i}\sum\limits_j\alpha_j(\alpha_j+1)\text{   otherwise.}\nonumber
\end{align}
\indent For any $m$ with $0\leq m<V^{rel}_0$, consider the minimal integer $i$ such that there exists an $\alpha$ that satisfies
\begin{equation}\label{eq:7}
rp-|\sigma|_1-2\sum\limits_j\sigma_j\alpha_j=\begin{cases}2i+1\text{   if }r\text{ is even and }p\text{ is odd,}\\2i\text{   otherwise,}\end{cases}
\end{equation}
\noindent and
$$\sum\limits_j\alpha_j(\alpha_j+1)\leq 2m.$$
\noindent Notice that all entries of such an $\alpha$ must be non-negative because if there is any negative entry $\alpha_j$, we can replace it with $-1-\alpha_j$ which lowers $i$ but doesn't change $\sum\limits_j\alpha_j(\alpha_j+1)$. Hence, such an $\alpha$ must be in $S_{\leq m}$. (Note that the entries being non-negative is in the definition of $\alpha\in S_{\leq m}$.)

From the argument above, we see that if such an $i$ is non-negative, we have $i=min\{t \mid V^{rel}_t\leq m\}$.

Now we argue that $m<V^{rel}_0$ implies such an $i$ being non-negative. Assume the contrary and suppose such a minimal $i$ is negative. Let $k$ be the smallest positive integer such that there exists $\alpha\in S_{\leq m}$ with
$$rp-|\sigma|_1-2\sum\limits_j\sigma_j\alpha_j=\begin{cases}2(-k)+1\text{   if }r\text{ is even and }p\text{ is odd,}\\2(-k)\text{   otherwise.}\end{cases}$$
Let $l$ be maximal such that $\alpha_l\neq 0$. If $k>\sigma_l$, we can decrease $\alpha_l$ by 1 to produce a new $\alpha$ that is still in $S_{\leq m}$ but $k$ decreases by $\sigma_l$, contradicting with the minimality of $k$. Hence, we know that $k\leq\sigma_l$. Since $\sigma$ is a changemaker, there exists $A\subseteq\{l,\dots,n\}$ such that $k=\sum\limits_{j\in A}\sigma_j$. Consider $\alpha'$ where
$$\alpha'_j:=\begin{cases}\alpha_j-1\text{   if }j\in A, \\ \alpha_j\text{   otherwise.}\end{cases}$$
\noindent Then $\alpha'\in\mathbb{Z}_{\geq-1}^{n+1}$ satisfies $\sum\limits_j\alpha'_j(\alpha'_j+1)\leq m$ and
$$rp-|\sigma|_1-2\sum\limits_j\sigma_j\alpha'_j=\begin{cases}1\text{   if }r\text{ is even and }p\text{ is odd,}\\0\text{   otherwise.}\end{cases}$$
\indent Equation (\ref{eq:6}) implies that $V_0^{rel}\leq m$, which contradicts $m<V^{rel}_0$. Hence, we can conclude that $m<V^{rel}_0$ implies such an $i$ being non-negative, and therefore we have $i=min\{t \mid V^{rel}_t\leq m\}$. So, we can conclude that
$$min\{t \mid V^{rel}_t\leq m\}=\begin{cases}\displaystyle\min_{\alpha\in S_{\leq m}}\frac{1}{2}\left(rp-|\sigma|_1-1-2\sum\limits_j\sigma_j\alpha_j\right)\text{   if }r\text{ is even and }p\text{ is odd,}\\ \displaystyle\min_{\alpha\in S_{\leq m}}\frac{1}{2}\left(rp-|\sigma|_1-2\sum\limits_j\sigma_j\alpha_j\right)\text{   otherwise.}\end{cases}$$
\noindent Setting $m=0$ gives statement (2) of Lemma~\ref{linear system}. By monotonicity of relevant coefficients, we know that $T^{rel}_m=\nu^{+rel}-min\{t \mid V^{rel}_t\leq m\}$. That gives statement (1) of Lemma~\ref{linear system}.

This argument does not fully apply to the $m=V^{rel}_0$ case because of the issue of $i$ potentially being negative, causing it to not represent a relevant coefficient. However, we can still apply (\ref{eq:6}) to see that there exists some $\alpha\in\mathbb{Z}^{n+1}$ such that $2V^{rel}_0=\sum\limits_j\alpha_j(\alpha_j+1)$ and 
$$rp-|\sigma|_1-2\sum\limits_j\sigma_j\alpha_j=\begin{cases}1\text{   if }r\text{ is even and }p\text{ is odd,}\\0\text{   otherwise.}\end{cases}$$
We see that $\nu^{+rel}=T^{rel}_{V^{rel}_0}=\sigma\cdot\alpha$ for this particular $\alpha$. Hence,
$$\nu^{+rel}=T^{rel}_{V^{rel}_0}\leq\displaystyle\max_{\substack{\alpha\in\mathbb{Z}^{n+1} \\ 2V^{rel}_0=\sum\limits_j\alpha_j(\alpha_j+1)}}\sigma\cdot\alpha.$$
If $\alpha$ has any negative entry $\alpha_j$, we can replace it with $-1-\alpha_j$ which increases $\sigma\cdot\alpha$ but keeping $\sum\limits_j\alpha_j(\alpha_j+1)$ unchanged. Hence, we can take the max to be over $\alpha\in S_{V^{rel}_0}$, which gives statement (3) of Lemma~\ref{linear system}.
\end{proof}

\begin{Rem1}
Since relevant coefficients can skip numbers, $2m$ may not equal to $V^{rel}_i$ for any $i$, and hence we need to introduce $S_{\leq m}$ in this paper instead of just $S_m$ as in \cite[\S 2]{mccoy}.
\end{Rem1}

\begin{Rem2}
The reason why the $m=V^{rel}_0$ case was dealt with separately in the proof is because of the issue of $i$ potentially being negative. If there is some condition that implies $\max\limits_{\alpha\in S_{V^{rel}_0}}\sigma\cdot\alpha-\max\limits_{\alpha\in S_{V^{rel}_0-1}}\sigma\cdot\alpha$ to be at most 1, since the minimal $i$ for $m=V^{rel}_0-1$ in (\ref{eq:7}) is positive, such an $i$ for $m=V^{rel}_0$ must be non-negative. Hence, under such condition, statement (1) of Lemma~\ref{linear system} would also apply to $m=V^{rel}_0$. In Section 6, we will show that $r\geq 2$ is such a condition.
\end{Rem2}

Most of the remaining content of this section is for proving Theorem~\ref{main result 2}.

\begin{theorem*}[\ref{main result 2}]
Let $K$ be a knot in $S^3$ with $\nu^+(K)\neq 0$. Then,\\
(1) When $r=1$, there are at most 3 $(r,p)$-lensbordant surgeries on $K$. \\
(2) When $r$ is even, there are at most 2 $(r,p)$-lensbordant surgeries on $K$.\\
(3) When $r\geq 3$ is odd, there is at most 1 $(r,p)$-lensbordant surgeries on $K$.
\end{theorem*}

Many ideas that go into the proof come from \cite[\S 2]{mccoy}.

\begin{Lemma}\label{r=1 case}
Statement (1) of Theorem~\ref{main result 2} is true.
\end{Lemma}

\begin{proof}
When $r=1$, the conditions in Lemma~\ref{linear system} become the conditions in \cite[Lemma 2.10]{mccoy}. Therefore, arguments in \cite[\S 2-4]{mccoy} apply directly.
\end{proof}

For statement (2) and (3), notice that specifying $K$ and with the value of $r$, and also specifying the parity of $p$ when $r$ is even, together determines all the relevant coefficients. So, it suffices to show that the relevant coefficients and the value of $r$, and the parity of $p$ when $r$ is even, together determines $\sigma$. This is because $\sigma$ determines $p$, which determines the surgery slope $r^2p$ when we know the value of $r$.

We deal with the trivial cases first.

\begin{Lemma} \label{t0<2}
If $V^{rel}_0=0$ and $r\geq 2$, then $r=2$ and $\sigma=(1)$.

If $V^{rel}_0=1$ and $r\geq 2$, then it must be one of the following cases:\\
$r=2$ and $\sigma=(1,1) \ or \ (1,1,1)$\\
$r=3$ and $\sigma=(1)$\\
$r=4$ and $\sigma=(1)$.

If $V^{rel}_0=2$ and $r\geq 2$ and $\sigma_0=2$ and $\sigma_1=1$, then $r=2$ and $\sigma=(2,1)$.
\end{Lemma}

\begin{proof}
If $V^{rel}_0=0$, that means there are no non-zero relevant coefficients. Therefore,
$$0=\nu^{+rel}=\begin{cases}-\dfrac{1}{2}+\dfrac{1}{2}(rp-|\sigma|_1)\text{   if }r\text{ is even and }p\text{ is odd,}\\\dfrac{1}{2}(rp-|\sigma|_1)\text{   otherwise.}\end{cases}$$
\noindent Since $p\geq |\sigma|_1$, given that $r\geq 2$, the only way this expression can be 0 is when $r=2$ and $p=|\sigma|_1=1$.

If $V^{rel}_0=1$, that means all non-zero relevant coefficients are $1$. Hence we have $\nu^{+rel}=T^{rel}_1$. Now we have
$$0\neq\nu^{+rel}=T^{rel}_1\leq\max_{\alpha\in S_1}\sigma\cdot\alpha=\sigma_0.$$
Hence we have
$$\begin{cases}-\dfrac{1}{2}+\dfrac{1}{2}(r\sigma_0^2-\sigma_0)\leq-\dfrac{1}{2}+\dfrac{1}{2}(rp-|\sigma|_1)=g^{rel}\leq\sigma_0\text{   if }r\text{ is even and }p\text{ is odd,}\\ \dfrac{1}{2}(r\sigma_0^2-\sigma_0)\leq \dfrac{1}{2}(rp-|\sigma|_1)=g^{rel}\leq\sigma_0\text{   otherwise.}\end{cases}$$
Rearranging, we have
$$r\sigma_0\leq\begin{cases}3+\frac{1}{\sigma_0}\text{   if }r\text{ is even and }p\text{ is odd,}\\ 3\text{   otherwise.}\end{cases}$$
\noindent So, we have $r\leq 4$. Given that $r\geq 2$, we also have $\sigma_0=1$. So we can conclude that $\nu^{+rel}\leq 1$ and all entries of $\sigma$ are $1$. Since $\nu^{+rel}\neq 0$, we have $\nu^{+rel}=1$. By considering the expression for $\nu^{+rel}$ in terms of $r,p,|\sigma|_1$, we can see that the only possibilities are as stated in the lemma.

If $V^{rel}_0=2$ and $r\geq 2$ and $\sigma_0=2$ and $\sigma_1=1$, then
$$\nu^{+rel}=T^{rel}_2\leq\max_{\alpha\in S_2}\sigma\cdot\alpha=\sigma_0+\sigma_1=3.$$
Let $x$ be the number of 1's in $\sigma$. Then
\begin{align*}
\nu^{+rel}&\geq -\dfrac{1}{2}+\dfrac{1}{2}(rp-|\sigma|_1)\\
&= -\dfrac{1}{2}+\dfrac{1}{2}\left(r(4+x)-(2+x)\right)\\
&=2r-1+\dfrac{(r-1)x-1}{2}.
\end{align*}
Hence, we must have $r=2$ and $x=1$.
\end{proof}

Theorem~\ref{main result 3} comes as a corollary.

\begin{theorem*}[\ref{main result 3}]
Let $K$ be a knot in $S^3$ with $\nu^+(K)=0$ and let $m$ be a positive integer. If the $m$-surgery on $K$ is lensbordant, then it is smoothly rational homology cobordant to $L(m,m-1)$.
\end{theorem*}

\begin{proof}[Proof of Theorem~\ref{main result 3}]
When $\nu^+=0$, Lemma~\ref{t0<2} implies that either $r=1$, or $r=2$ with $\sigma=(1)$.

The changemaker $\sigma=(1)$ represents $S^3$. Hence, in the $r=2$ with $\sigma=(1)$ case, we have $m=4$, and $Y$ is smoothly rational homology cobordant to $S^3$. Since $S^3$ is smoothly rational homology cobordant to $L(4,3)$, $Y$ is also smoothly rational homology cobordant to $L(4,3)$.

When $r=1$, the conditions in Lemma~\ref{linear system} become the conditions in \cite[Lemma 2.10]{mccoy}. Therefore, when $\nu^+(K)=0$, all entries of $\sigma$ are 1's \cite[Lemma 2.12]{mccoy}. These changemakers correspond to $L(m,m-1)$.
\end{proof}

Given Lemma~\ref{r=1 case} and \ref{t0<2}, \textbf{from this point onwards, until the end of this section, we assume that $r\geq 2$ and $V^{rel}_0\geq 2$}.

Now we can define
$$\mu:=\min_{1\leq i<V^{rel}_0}(T^{rel}_i-T^{rel}_{i-1}).$$
Note that with $r\geq 2$, $\mu$ can be 0 here. This is very different from what McCoy did in \cite{mccoy} where $\mu$ must be at least 1.

\begin{Lemma} \label{t0 lower bound}
$$V^{rel}_0\geq\begin{cases}\frac{1}{8}((r^2p-2r+1)-(n+1))\text{   if }r\text{ is even and }p\text{ is odd}\\ \frac{1}{8}(r^2p-(n+1))\text{   otherwise}\end{cases}$$
\end{Lemma}

\begin{proof}
At the start of the proof of Lemma~\ref{linear system}, we see that
$$-8V^{rel}_0=\mathfrak{c}^2+(n+1)$$
for some characteristic vector $\mathfrak{c}$ satisfying
$$\langle \mathfrak{c},\sigma\rangle+rp=\begin{cases}1\text{   if }r\text{ is even and }p\text{ is odd,}\\0\text{   otherwise.}\end{cases}$$
By Cauchy-Schwarz inequality and the fact that $\sigma^2=-p$, we have
$$|\mathfrak{c}^2|=\frac{|\mathfrak{c}^2||\sigma^2|}{p}\geq\frac{|\langle \mathfrak{c},\sigma\rangle|^2}{p}\geq\begin{cases}\frac{(rp-1)^2}{p}\text{   if }r\text{ is even and }p\text{ is odd,}\\ \frac{(rp)^2}{p}\text{   otherwise.}\end{cases}$$
Simplifying the expressions and considering the fact that $|\mathfrak{c}^2|\in\mathbb{Z}$, we have
$$|\mathfrak{c}^2|\geq\begin{cases}r^2p-2r+1\text{   if }r\text{ is even and }p\text{ is odd}\\ r^2p\text{   otherwise}\end{cases}$$
The result follows by substituting this back into $-8V^{rel}_0=\mathfrak{c}^2+(n+1)$.
\end{proof}

\begin{Lemma} \label{mu at most 1}
If $\mu\geq 2$, then $V_0^{rel}=2$ and $\sigma_0=2$ and $\sigma_1=1$.
\end{Lemma}

\begin{proof}
Suppose $\mu\geq 2$. We define
$$M:=\begin{cases}\lceil\frac{1}{8}((r^2p-2r+1)-(n+1))\rceil-1\text{   if }r\text{ is even and }p\text{ is odd,}\\ \lceil\frac{1}{8}(r^2p-(n+1))\rceil-1\text{   otherwise,}\end{cases}$$
which is an analogue of $N-1$ in \cite[\S 2]{mccoy}.

By Lemma~\ref{t0 lower bound}, we have $V_0^{rel}>M$. Hence, by Lemma~\ref{linear system}, we know that
$$T^{rel}_M=\displaystyle\max_{\alpha\in S_{\leq M}}\sigma\cdot\alpha.$$
Take such an $\alpha$.

First of all, we argue that $\alpha\in S_M$ (without the $\leq$). If $\alpha\not\in S_M$, then we must have $M\geq 1$ and $\max\limits_{\alpha\in S_{\leq M}}\sigma\cdot\alpha=\max\limits_{\alpha\in S_{\leq M-1}}\sigma\cdot\alpha$. This implies $T^{rel}_M=T^{rel}_{M-1}$, which contradicts $\mu\geq 2$. Hence, $\alpha\in S_M$.

For any index $l$ with $\alpha_l>0$, if we consider $\alpha^{'}$ defined with $\alpha^{'}_i=\begin{cases}\alpha_i\text{ if }i\neq l,\\ \alpha_i-1\text{ if }i=l,\end{cases}$ we have $\alpha^{'}\in S_{\leq M-\alpha_l}$. So, we have
$$T^{rel}_{M-\alpha_l}\geq\alpha^{'}\cdot\sigma=\alpha\cdot\sigma-\sigma_l=T^{rel}_M-\sigma_l.$$
Hence,
$$\sigma_l\geq T^{rel}_M-T^{rel}_{M-\alpha_l}\geq\alpha_l\mu.$$
Since $\mu\geq 2$, from our discussion above, for all index $l$ we have $\sigma_l\geq 2\alpha_l$ (this also holds when $\alpha_l=0$ because the right hand side becomes $0$). Hence,
\begin{align*}
M&=\dfrac{1}{2}\displaystyle\sum_{l}\alpha_l(\alpha_l+1)\\
&\leq\dfrac{1}{2}\displaystyle\sum_{l}\dfrac{\sigma_l}{2}\left(\dfrac{\sigma_l}{2}+1\right)\\
&=\dfrac{1}{8}\displaystyle\sum_{l}\sigma_l(\sigma_l+2)\\
&=\dfrac{1}{8}(p+2|\sigma|_1).
\end{align*}
Combining this with the definition of $M$, we have
$$\dfrac{1}{8}(p+2|\sigma|_1)\geq\begin{cases}\left\lceil\frac{1}{8}((r^2p-2r+1)-(n+1))\right\rceil-1\text{   if }r\text{ is even and }p\text{ is odd,}\\ \left\lceil\frac{1}{8}(r^2p-(n+1))\right\rceil-1\text{   otherwise.}\end{cases}$$
Since $r\geq 2$ and $|\sigma|_1\geq(n+1)$ (because $\sigma$ has no $0$ entries), we have
$$\dfrac{1}{8}(p+2|\sigma|_1)\geq\begin{cases}\frac{1}{8}((4p-4+1)-|\sigma|_1)-1\text{   if }r\text{ is even and }p\text{ is odd,}\\ \frac{1}{8}(4p-|\sigma|_1)-1\text{   otherwise.}\end{cases}$$
Rearranging this, we get
$$\begin{cases}\dfrac{11}{3}\geq p-|\sigma|_1\text{   if }r\text{ is even and }p\text{ is odd,}\\[3ex] \dfrac{8}{3}\geq p-|\sigma|_1\text{   otherwise.}\end{cases}$$
Since $p-|\sigma|_1$ is an integer, this implies
$$3\geq\displaystyle\sum_{l}(\sigma_l^2-\sigma_l).$$
Therefore, the entries of $\sigma$ is either a 2 followed by a bunch of 1's or simply all 1's.

Since $V^{rel}_0\geq 2$, we have $T^{rel}_0=0$ and $T^{rel}_1=\sigma_0$. Also, if $V^{rel}_0\geq 3$, we have $T^{rel}_2=\sigma_0+\sigma_1$. Therefore, the only way for $\mu$ to be at least 2 is when $V^{rel}_0=2$ and $\sigma$ is a 2 followed by a bunch of 1's.
\end{proof}

Now we can prove Theorem~\ref{main result 2}. Statement (1) of Theorem~\ref{main result 2} has already been proven in Lemma~\ref{r=1 case}. So we only need to prove statements (2) and (3) of Theorem~\ref{main result 2}.

\begin{proof}[Proof of statement (2) and (3) of Theorem~\ref{main result 2}]
Fix the knot $K$. Fix $r\geq 2$. Also, if $r$ is even, we fix the parity of $p$ (odd or even).

We want to show that these together determines $p$, which determines the surgery slope $r^2p$ since $r$ is fixed.

These data determine the relevant coefficients. So, it suffices to show that the relevant coefficients together with the value of $r$ determines $\sigma$ (which determines $p$).

The algorithm McCoy described in the proof of \cite[Lemma 2.16]{mccoy} shows that statement (1) of our Lemma~\ref{linear system} determines all entries of $\sigma$ strictly greater than $\mu$. Together with Lemma~\ref{mu at most 1}, this implies that all entries of $\sigma$ except for the 1's are determined, except for a special case which was addressed in Lemma~\ref{t0<2}. Given that $r\geq 2$, statement (2) of Lemma~\ref{linear system} determines the number of 1's.
\end{proof}

\begin{figure}
\centering
\includegraphics[width=0.95\textwidth]{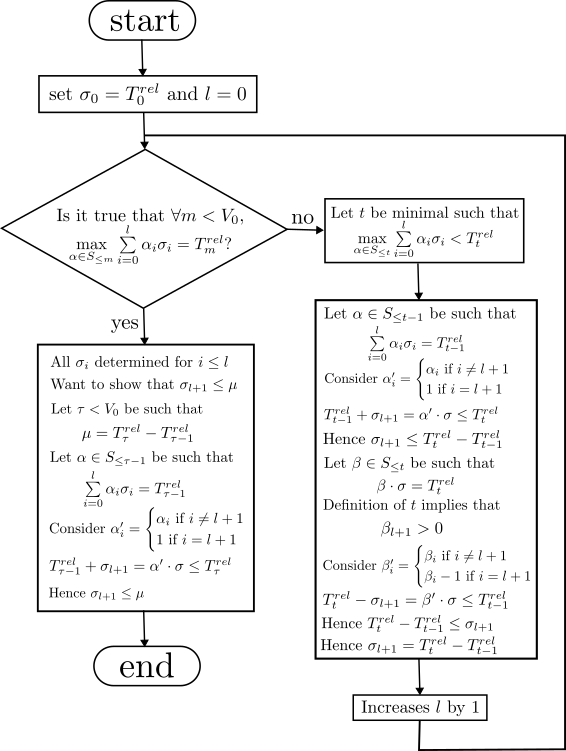}
\caption{Flowchart showing McCoy's algorithm and argument, but in notation used in this paper.}
\end{figure}

\FloatBarrier

\section{Some bounds on the surgery slope \texorpdfstring{$r^2p$}{r\texttwosuperior p}}
In this section, we establish some inequalities involving the surgery slope $r^2p$.

\begin{theorem*}[\ref{main result 4}]
Suppose there is an $(r,p)$-lensbordant surgery on a knot $K$ in $S^3$ with associated changemaker $\sigma$. Then
$$2\nu^+(K)+2(r-1)\geq r^2p-r|\sigma|_1\geq 2\nu^+(K).$$
\end{theorem*}

\begin{proof}
When $\nu^{+rel}=0$, Lemma~\ref{t0<2} implies that either $r=2$ with $\sigma=(1)$, or $r=1$. When $r=2$ and $\sigma=(1)$, the inequality becomes $2\geq 2\geq 0$. When $r=1$, $V_0=V^{rel}_0=0$, and hence $\nu^+=0$. As discussed in the proof of Theorem~\ref{main result 3}, $\sigma$ must be $(1,\dots,1)$ in this case, and the inequality becomes $0\geq 0\geq 0$.

Now we assume that $\nu^{+rel}\neq 0$.

\underline{Case 1: $r$ is odd or $p$ is even}

We have $\nu^{+rel}=\dfrac{1}{2}(rp-|\sigma|_1)$. Hence,
$$V^{rel}_{\tfrac{1}{2}(rp-|\sigma|_1)}=0 \quad\text{and}\quad V^{rel}_{\tfrac{1}{2}(rp-|\sigma|_1)-1}\neq 0.$$
So we have 
$$V_{\tfrac{1}{2}(r^2p-r|\sigma|_1)}=0 \quad\text{and}\quad V_{\tfrac{1}{2}(r^2p-r|\sigma|_1)-r}\neq 0.$$
This implies
$$\dfrac{1}{2}(r^2p-r|\sigma|_1)\geq\nu^+>\dfrac{1}{2}(r^2p-r|\sigma|_1)-r.$$
Rearranging, we have
$$2\nu^++2r>r^2p-r|\sigma|_1\geq 2\nu^+.$$
Since $p\equiv|\sigma|_1$ (mod 2), the value $r^2p-r|\sigma|_1$ must be an even integer. Hence
$$2\nu^++2(r-1)\geq r^2p-r|\sigma|_1\geq 2\nu^+.$$
\underline{Case 2: $r$ is even and $p$ is odd}\\
We have $\nu^{+rel}=-\dfrac{1}{2}+\dfrac{1}{2}(rp-|\sigma|_1)$. Hence,
$$V^{rel}_{-\tfrac{1}{2}+\tfrac{1}{2}(rp-|\sigma|_1)}=0 \quad\text{and}\quad V^{rel}_{-\tfrac{3}{2}+\tfrac{1}{2}(rp-|\sigma|_1)}\neq 0.$$
So we have 
$$V_{\tfrac{1}{2}(r^2p-r|\sigma|_1)}=0 \quad\text{and}\quad V_{\tfrac{1}{2}(r^2p-r|\sigma|_1)-r}\neq 0.$$
The rest of the proof is the same as Case 1.
\end{proof}

Now we prove Corollary~\ref{main result 5}.

\begin{corollary*}[\ref{main result 5}]
Suppose there is an $(r,p)$-lensbordant surgery on a knot $K$ in $S^3$ with associated changemaker $\sigma$.\\
(1) If $r\geq 3$, then $\dfrac{2r^2}{(r-1)(r-2)}\nu^+(K)\geq r^2p\geq 2\nu^+(K)+r$.\\
(2) If $r\geq 2$, then $4\nu^+(K)+5\geq r^2p\geq 2\nu^+(K)+r$.
\end{corollary*}

\begin{proof}
Since $|\sigma|_1\geq 1$, Theorem~\ref{main result 4} implies that $r^2p\geq 2\nu^++r$.

Since $p\geq|\sigma|_1$, Theorem~\ref{main result 4} implies that $2\nu^++2(r-1)\geq r(r-1)p$. So we have
\begin{equation}\label{eq:8}
\dfrac{2r}{r-1}\nu^++2r\geq r^2p.
\end{equation}
Then we have
$$\dfrac{r-2}{r}r^2p=r^2p-\dfrac{2}{r}r^2p\leq\left(\dfrac{2r}{r-1}\nu^++2r\right)-2r=\dfrac{2r}{r-1}\nu^+.$$
Statement (1) follows.

When $r\geq 6$, $\dfrac{2r^2}{(r-1)(r-2)}<4$. So, we only need to check statement (2) for $r=2,3,4,5$.\\
\\
When $r=2$, (\ref{eq:8}) implies $r^2p\leq 4\nu^++4$.\\
\\
By Lemma~\ref{t0<2}, when $r\geq 3$, we have $\nu^+\geq 1$.\\
When $r=3$, (\ref{eq:8}) implies $r^2p\leq 3\nu^++6\leq 4\nu^++5$.\\
\\
When $r=4$ and $\nu^+\geq 3$, (\ref{eq:8}) implies $r^2p\leq \tfrac{8}{3}\nu^++8=4\nu^++4-\tfrac{4}{3}(\nu^+-3)\leq 4\nu^++4$.\\
When $r=4$ and $\nu^+\leq 2$, statement (1) implies $r^2p\leq\tfrac{16}{3}\nu^+=4\nu^++\tfrac{8}{3}-\tfrac{4}{3}(2-\nu^+)\leq 4\nu^++\tfrac{8}{3}$.\\
\\
When $r=5$ and $\nu^+\geq 4$, (\ref{eq:8}) implies $r^2p\leq \tfrac{5}{2}\nu^++10=4\nu^++4-\tfrac{3}{2}(\nu^+-4)\leq 4\nu^++4$.\\
When $r=5$ and $\nu^+\leq 3$, statement (1) implies $r^2p\leq\tfrac{25}{6}\nu^+=4\nu^++\tfrac{1}{2}-\tfrac{1}{6}(3-\nu^+)\leq 4\nu^++\tfrac{1}{2}$.
\end{proof}

\FloatBarrier

\section{Structure of relevant coefficients}
In the introduction, the remark after the statement of Theorem~\ref{imprecise theorem} states that the surgery slope $r^2p$ is very close to $8V_0(K)$. The precise statement is the following theorem:

\begin{Th} \label{main result 6}
(1) If $r,p$ are both even, then $r^2p=8V_0(K)$.\\
(2) If $r$ is even and $p$ is odd, then $8V_0(K)-2r\leq r^2p\leq 8V_0(K)+2r$.\\
(3) If $r$ is odd, then $8V_0(K)-3odd(\sigma)\leq r^2p\leq 8V_0(K)+odd(\sigma)$,\\
where $odd(\sigma)$ is the number of odd entries in $\sigma$.\\
\\
Furthermore, in the case when $r$ is odd, $r^2p=8V_0(K)+odd(\sigma)$ if and only if the even entries of $\sigma$ can be partitioned into two sets with equal sum.
\end{Th}

Note that Theorem~\ref{imprecise theorem} is statement (1) of Theorem~\ref{main result 6}. The aim of this section is to gain a deeper understanding on relevant coefficients, and then prove Theorem~\ref{main result 6}.

\begin{table}
\caption{Table describing Theorem~\ref{main result 6}.}
\begin{tabular}{ c|c|c } 
 & $r$ even & $r$ odd \\ 
\hline
$p$ even & $r^2p=8V_0$ & $8V_0-3odd(\sigma)\leq r^2p\leq 8V_0+odd(\sigma)$ \\ 
       &             &  upper bound is equality if and only if \\
\cline{1-2}
$p$ odd & $8V_0-2r\leq r^2p\leq 8V_0+2r$ & the even entries of $\sigma$ can be partitioned \\
      &                                & into two sets with equal sum
\end{tabular}\\
\end{table}

The $T^{rel}_m$ in statement (1) of Lemma~\ref{linear system} is defined with the relevant coefficients. Besides $m$, it also depends on the knot $K$, the value of $r$, and the parity of $p$ if $r$ is even. However, the right hand side $\displaystyle\max_{\alpha\in S_{\leq m}}\sigma\cdot\alpha$ only depends on $m$ and $\sigma$. Hence, it makes sense to think of that as some coefficients coming from $\sigma$ itself.

\FloatBarrier

\begin{Def} \label{sigma rel coeff}
For $m\geq 0$, we define
$$T^{\sigma}_m:=\displaystyle\max_{\alpha\in S_{\leq m}}\sigma\cdot\alpha,$$
and we define the relevant coefficients $V^{\sigma}_i$ (for $i>0$) of the changemaker vector $\sigma$ to be the monotonic \textbf{increasing} sequence of positive integers such that $T^{\sigma}_m$ is the number of relevant coefficients less than or equal to $m$. Equivalently,
$$V^{\sigma}_i:=\min\{m \mid i\leq T^{\sigma}_m\}, \quad i\in\mathbb{Z}_{>0}.$$
\end{Def}

We also define $V^{\sigma}_0$ to be $0$, but we do not call it a relevant coefficient to avoid making the above definition more complicated.

Now statement (1) and (3) of Lemma~\ref{linear system} can be rewritten as
$$T^{rel}_m=T^{\sigma}_{m}\text{ for all }m<V^{rel}_0$$
and
$$T^{rel}_{V^{rel}_0}\leq T^{\sigma}_{V^{rel}_0}.$$

Together they imply
$$V^{rel}_i=V^{\sigma}_{\nu^{+rel}-i}\text{ for all }0\leq i<\nu^{+rel}.$$
\indent Recall that the definition of $S_{\leq m}$ is
$$S_{\leq m}:=\{\alpha\in\mathbb{Z}^{dim(\sigma)}_{\geq 0} \mid \Sigma\alpha_j(\alpha_j+1)\leq 2m\}.$$
Notice that $\tfrac{1}{2}\alpha_j(\alpha_j+1)$ is exactly $1+2+\dots+\alpha_j$. Hence, we can interpret $T^{\sigma}_m$ as the answer to the following combinatorial problem:

\textit{Consider the following game. There is a shop where you can spend coins to buy items. The available items are labeled $I_0,\dots,I_n$. You can buy any non-negative integer amount of each of the items. For each individual item, it costs 1 coin to buy the first one, 2 coins to buy the second one, 3 coins to buy the third one, etc. After finishing doing all the purchases, each $I_i$ you own gives you $\sigma_i$ points. If you start with $m$ coins, what is the highest amount of points you can get?}

A way to approach this game is to consider the coin cost per point in each purchase option. Purchasing the item $I_i$ for the $k^{th}$ time costs $k$ coins to obtain $\sigma_i$ points, hence it is a $\dfrac{k}{\sigma_i}$ coin per point purchase. One intuitive thing to do is the use a greedy algorithm by always making the purchase with the lowest coin cost per point among all available purchase options until running out of coins.

This greedy algorithm does not always give the optimal solution. At the end, you might have some coins left that is not enough to purchase the next lowest coin per point purchase option and are forced to purchase a cheaper option with a higher cost per point, or simply having some leftover coins with no available purchase options.

However, consider a modified version of the game where you are allowed to make fractional purchases under the same coin per point. For example, being able to buy $(4+\tfrac{2}{3})$ copies of $I_i$ by spending $(1+2+3+4+\tfrac{2}{3}\times 5)$ coins, gaining $(4+\tfrac{2}{3})\sigma_i$ points.

In this modified version of the game, a greedy algorithm will not result in any leftover coins, and all the purchased options have a lower or equal coin per point than any unpurchased option. Hence, a greedy algorithm must give the optimal solution.

\begin{Def}
We call a greedy algorithm on the modified version of the game a \textit{rational greedy algorithm}. We write $T^{\sigma,\mathbb{Q}}_m$ to denote the amount of points obtained with $m$ coins under the rational greedy algorithm.

Given $m,\sigma$, we say that the resulting rational greedy algorithm is \textit{up to $x\in\mathbb{Q}$ coin per point} if the resulting rational greedy algorithm makes a non-zero amount of purchase options with $x$ coin per point and no purchase options with $>x$ coin per point.
\end{Def}

Because of the way rational greedy algorithm works, a rational greedy algorithm up to $x$ coin per point must make all available purchases that are $<x$ coin per point.

Also, given $\sigma$ and $m$, if a rational greedy algorithm is up to $x$ coin per point, then all other optimal solutions to the modified version of the game must also take the same form: purchasing all available purchases that are $<x$ coin per point, and not purchasing any options that are $>x$ coin per point. This is because any other way of spending all $m$ coins will result in having a higher average coin per point, and thus a lower total point at the end.

Hence, in this modified version of the game, all optimal solutions must be rational greedy algorithms up to $x$ coin per cost. The only difference allowed in different optimal solutions is picking different $x$ coin per point purchase options.

\begin{Def}
If an optimal solution in the modified version of the game does not include fractional purchases, we say that the solution is \textit{integral}. Equivalently, an integral optimal solution in the modified version of the game is a solution that is valid in the original version of the game.

If an optimal solution in the modified version of the game is a rational greedy algorithm up to $x$ coin per point that purchases all options that are $x$ coin per point, we say that the solution is \textit{full}.
\end{Def}

Clearly, a full rational greedy algorithm must also be integral.

All purchase options in the original version of the game are valid purchase options in the modified version of the game. Hence, $T^{\sigma,\mathbb{Q}}_m\geq T^{\sigma}_m$. In particular, $T^{\sigma,\mathbb{Q}}_m=T^{\sigma}_m$ if and only if there exist an integral implementation of the rational greedy algorithm to the modified version of the game.

In this section, we will use these ideas to help investigating the structure of relevant coefficients.

\begin{Lemma}\label{multiples of p on T}
For any positive integer $x$,
$$T^\sigma_{\tfrac{x^2p+x|\sigma|_1}{2}}=xp.$$
\end{Lemma}

\begin{proof}
Consider the full rational greedy algorithm up to $x$ coin per point. So, for each $i$, it purchases $x\sigma_i$ copies of the item $I_i$. The total amount of coins spent is
\begin{align*}
\dfrac{1}{2}\sum\limits_i(x\sigma_i)(x\sigma_i+1)&=\dfrac{1}{2}x^2\sum\limits_i\sigma_i+\dfrac{1}{2}x\sum\limits_i\sigma_i\\
&=\dfrac{x^2p+x|\sigma|_1}{2},
\end{align*}
and the number of points gained is
$$\sum\limits_i(x\sigma_i)\sigma_i=xp.$$
Hence, $T^{\sigma,\mathbb{Q}}_{\tfrac{x^2p+x|\sigma|_1}{2}}=xp$. Since the greedy algorithm is full, it must be integral. Therefore,
$$T^{\sigma,\mathbb{Q}}_{\tfrac{x^2p+x|\sigma|_1}{2}}=T^\sigma_{\tfrac{x^2p+x|\sigma|_1}{2}}.$$
\end{proof}

\begin{Cor} \label{multiples of p on V}
For any positive integer $x$,
$$V^\sigma_{xp}=\dfrac{x^2p+x|\sigma|_1}{2}.$$
\end{Cor}

\begin{proof}
From Lemma~\ref{multiples of p on T}, we know there are exactly $xp$ relevant coefficients less than or equal to $\dfrac{x^2p+x|\sigma|_1}{2}$. Since the relevant coefficients of a changemaker vector are monotonic increasing, we know that
$$V^\sigma_{xp}\leq\dfrac{x^2p+x|\sigma|_1}{2}.$$
To prove $V^\sigma_{xp}=\dfrac{x^2p+x|\sigma|_1}{2}$, it suffices to show that
$$T^\sigma_{\tfrac{x^2p+x|\sigma|_1}{2}-1}<xp,$$
i.e. there are strictly less than $xp$ relevant coefficients that are less than or equal to $\dfrac{x^2p+x|\sigma|_1}{2}-1$ (which implies that $V^\sigma_{xp}$ cannot be less than or equal to $\tfrac{x^2p+x|\sigma|_1}{2}-1$).

This follows from the way Lemma~\ref{multiples of p on T} is proved. $xp$ points is obtained from a rational greedy algorithm with $\tfrac{x^2p+x|\sigma|_1}{2}$ coins. Since every coin contributes to some points under a rational greedy algorithm, strictly less than $xp$ points is obtained from a rational greedy algorithm with strictly less than $\tfrac{x^2p+x|\sigma|_1}{2}$ coins. Hence,
$$T^\sigma_{\tfrac{x^2p+x|\sigma|_1}{2}-1}\leq T^{\sigma,\mathbb{Q}}_{\tfrac{x^2p+x|\sigma|_1}{2}-1}<xp.$$
\end{proof}

Roughly speaking, Corollary~\ref{multiples of p on V} reveals that the relevant coefficents increase in a quadratic rate. The next thing we aim to show is that the first $p$ relevant coefficients of $\sigma$ determine the rest of them in a simple manner. Precisely, we want to show that for all positive integer $x$ and $0\leq k<p$,
\begin{equation}\label{eq:9}
V^{\sigma}_{xp+k}=\dfrac{x^2p+x|\sigma|_1}{2}+xk+V^{\sigma}_k.
\end{equation}
\noindent This is particularly useful for those who wish to do some programming on this topic because the size of $S_{\leq m}$ increases quickly with $m$, which can cause run-time and memory issues. Based on the data coming from some programming efforts related to this paper, when $m=100$, there are 157,452 $\alpha\in\mathbb{Z}^{something}_{>0}$ such that $\alpha_0\geq\alpha_1\geq\dots$ and $\sum\limits_j \alpha_j(\alpha_j+1)\leq 2m$. When $m=200$, the number becomes 13,552,451.

To help proving (\ref{eq:9}), we prove some more lemmas first.

\begin{Lemma}\label{multiples of p on T but minus}
For any positive integer $x$,
$$T^\sigma_{\tfrac{x^2p-x|\sigma|_1}{2}}=xp-|\sigma|_1.$$
\end{Lemma}

\begin{proof}
Consider the full rational greedy algorithm that purchases all options strictly less than $x$ coin per point. So, for each $i$, it purchases $x\sigma_i-1$ copies of the item $I_i$. The total amount of coins spent is
$$\dfrac{1}{2}\sum\limits_i(x\sigma_i-1)(x\sigma_i)=\dfrac{x^2p-x|\sigma|_1}{2},$$
and the number of points gained is
$$\sum\limits_i(x\sigma_i-1)\sigma_i=xp-|\sigma|_1.$$
The result follows from the fact that the rational greedy algorithm is full, and therefore integral.
\end{proof}

\begin{Cor} \label{multiples of p on V but minus}
$$V^\sigma_{xp-|\sigma|_1}=\dfrac{x^2p-x|\sigma|_1}{2}$$
\end{Cor}

\begin{proof}
Since Lemma~\ref{multiples of p on T but minus} came from an integral rational greedy algorithm, the argument used in the proof of Corollary~\ref{multiples of p on V} can be applied.
\end{proof}

\begin{Lemma} \label{a-b leq x}
For any positive integer $x$, if $a\leq\dfrac{x^2p+x|\sigma|_1}{2}$ and $T^{\sigma}_{a-1}<T^{\sigma}_{a}$ and $T^{\sigma}_b=T^{\sigma}_{a-1}$, then $a-b\leq x$.
\end{Lemma}

\begin{proof}
Let $\alpha\in S_{\leq b}$ be such that $\alpha\cdot\sigma=T^{\sigma}_b$.

To show that $a-b\leq x$, it suffices to show that there exists $\alpha'\in S_{\leq b+x}$ such that $\alpha'\cdot\sigma\geq T^{\sigma}_b+1$ because the existence of such an $\alpha'$ implies $T^{\sigma}_{\leq b+x}\geq T^{\sigma}_b+1>T^{\sigma}_b=T^{\sigma}_{a-1}$, which implies $b+x>a-1$, which implies $a-b\leq x$.

Let $i$ be minimal such that $\alpha_i<x\sigma_i$. Such an $i$ must exist, otherwise
$$\dfrac{1}{2}\sum\limits_j\alpha_j(\alpha_j+1)\geq\dfrac{x^2p+x|\sigma|_1}{2}\geq a>b,$$
which contradicts $\alpha\in S_{\leq b}$.

Since $\sigma$ is a changemaker, there exist $\sigma_{i_1},\dots,\sigma_{i_k}$ such that $\sigma_{i_1}+\dots+\sigma_{i_k}=\sigma_i-1$. (when $\sigma_i=1$, we set $k=0$ and this is an empty sum)

Define $\alpha'$ be such that
$$\alpha'_j=\begin{cases}\alpha_j+1\text{ if }j=i,\\
\alpha_j-1\text{ if }j=i_1,\dots,i_k,\\
\alpha_j\text{ otherwise.}
\end{cases}$$
Then
\begin{align*}
\alpha'\cdot\sigma&=\alpha\cdot\sigma+\sigma_i-\sigma_{i_1}-\dots-\sigma_{i_k}\\
&=\alpha\cdot\sigma+1\\
&=T^{\sigma}_b+1,
\end{align*}
and also
$$\alpha'\in S_{\leq b+(\alpha_i+1)-\alpha_{i_1}-\dots-\alpha_{i_k}}.$$
We complete the proof by noting that
\begin{align*}
b+(\alpha_i+1)-\alpha_{i_1}-\dots-\alpha_{i_k}&<b+(n\sigma_i+1)-n\sigma_{i_1}-\dots-n\sigma_{i_k}\\
&=b+x+1.
\end{align*}
Hence,
$$b+(\alpha_i+1)-\alpha_{i_1}-\dots-\alpha_{i_k}\leq b+x.$$
\end{proof}

\begin{Cor} \label{V diff upperbound}
For any positive integer $x$, if $0<a\leq xp$, then $V^\sigma_a-V^\sigma_{a-1}\leq x.$
\end{Cor}

\begin{proof}
If $V^\sigma_a=V^\sigma_{a-1}$, then the result is trivial. So, we assume $V^\sigma_a>V^\sigma_{a-1}$.

By Corollary~\ref{multiples of p on V}, we know that $V_a\leq\tfrac{x^2p+x|\sigma|_1}{2}$. Since there is at least one relevant coefficient with value $V^\sigma_a$ (namely $V^\sigma_a$ itself), we know that $T^\sigma_{V_a-1}<T^\sigma_{V_a}$. Also, since $V^\sigma_a>V^\sigma_{a-1}$, we have $T^\sigma_{V_{a-1}}=T^\sigma_{V_a-1}=a-1$. The result follows from Lemma~\ref{a-b leq x}.
\end{proof}

\begin{Lemma} \label{a-b geq x+1}
For any positive integer $x$, if $a>\dfrac{x^2p-x|\sigma|_1}{2}$ and $T^{\sigma}_a>T^{\sigma}_{a-1}$ and $T^{\sigma}_a-T^{\sigma}_b\geq 2$, then $a-b\geq x+1$.
\end{Lemma}

\begin{proof}
Let $\alpha\in S_{\leq a}$ be such that $\alpha\cdot\sigma=T^{\sigma}_a$.

To show that $a-b\geq x+1$, it suffices to show that $\exists\alpha'\in S_{\leq a-x}$ such that $\alpha'\cdot\sigma\geq T^{\sigma}_a-1$ because the existence of such an $\alpha'$ implies $T^{\sigma}_{\leq a-x}\geq T^{\sigma}_a-1\geq T^{\sigma}_{b}+1>T^{\sigma}_b$, which implies $a-x>b$, which implies $a-b\geq x+1$.

Let $i$ be minimal such that $\alpha_i\geq x\sigma_i$. Such an $i$ must exist, otherwise
\begin{align*}
\dfrac{1}{2}\sum\limits_j\alpha_j(\alpha_j+1)&\leq\dfrac{1}{2}\sum\limits_j(x\sigma_j-1)(x\sigma_j)\\
&=\dfrac{x^2p-x|\sigma|_1}{2}\\
&<a,
\end{align*}
which implies $\alpha\in S_{\leq a-1}$, which contradicts $T^{\sigma}_a>T^{\sigma}_{a-1}$.

Since $\sigma$ is a changemaker, there exist $\sigma_{i_1},\dots,\sigma_{i_k}$ such that $\sigma_{i_1}+\dots+\sigma_{i_k}=\sigma_i-1$. (when $\sigma_i=1$, we set $k=0$ and this is an empty sum)

Define $\alpha'$ be such that
$$\alpha'_j=\begin{cases}\alpha_j-1\text{ if }j=i,\\
\alpha_j+1\text{ if }j=i_1,\dots,i_k,\\
\alpha_j\text{ otherwise.}
\end{cases}$$
Then
\begin{align*}
\alpha'\cdot\sigma&=\alpha\cdot\sigma-\sigma_i+\sigma_{i_1}+\dots+\sigma_{i_k}\\
&=\alpha\cdot\sigma-1\\
&=T^{\sigma}_a-1,
\end{align*}
and also
$$\alpha'\in S_{\leq a-\alpha_i+(\alpha_{i_1}+1)+\dots+(\alpha_{i_k}+1)}.$$
We complete the proof by noting that
\begin{align*}
a-\alpha_i+(\alpha_{i_1}+1)+\dots+(\alpha_{i_k}+1)&\leq a-\alpha_i+n\sigma_{i_1}+\dots+n\sigma_{i_k}\\
&=a-n-(\alpha_i-n\sigma_i)\\
&\leq a-n.
\end{align*}
\end{proof}

\begin{Cor} \label{V diff lowerbound}
For any positive integer $x$, if $a\geq xp-|\sigma|_1$, then $V^\sigma_{a+1}-V^\sigma_a\geq x.$
\end{Cor}

\begin{proof}
By Lemma~\ref{multiples of p on T but minus}, there are only $xp-|\sigma|_1$ relevant coefficients less than or equal to $\tfrac{x^2p-x|\sigma|_1}{2}$. So, we know that
$$V^\sigma_{a+1}>\dfrac{x^2p-x|\sigma|_1}{2}.$$
Also, $V^\sigma_a$ and $V^\sigma_{a+1}$ are relevant coefficients less than or equal to $V^\sigma_{a+1}$ but not less than or equal to $V^\sigma_a-1$. Hence
$$T^\sigma_{V^\sigma_{a+1}}-T^\sigma_{V^\sigma_a-1}\geq 2.$$
Since $V^\sigma_{a+1}$ is a relevant coefficient, we have $T^\sigma_{V^\sigma_{a+1}}>T^\sigma_{V^\sigma_{a+1}-1}$. By Lemma~\ref{a-b geq x+1}, we know that
$$V^\sigma_{a+1}-(V^\sigma_a-1)\geq x+1.$$
\end{proof}

Another corollary of Lemma~\ref{a-b geq x+1} is the following:

\begin{Cor} \label{linear system equality}
When $r\geq 2$, statement (3) of Lemma~\ref{linear system} is an equality. In other words,
$$T^{rel}_{V_0^{rel}}=T^\sigma_{V_0^{rel}}.$$
\end{Cor}

\begin{proof}
When $r=2$ and $\sigma=(1)$, statement (2) of Lemma~\ref{linear system} says that $\nu^{+rel}=0$. Hence $V_0^{rel}=0$ and the result is trivial. Now we assume that it is not this special case.

By Remark 2 after the proof of Lemma~\ref{linear system}, it suffices to show that
\begin{equation}\label{eq:10}
T^\sigma_{V_0^{rel}}-T^\sigma_{V_0^{rel}-1}\leq 1.
\end{equation}
We assume that $T^\sigma_{V_0^{rel}}>T^\sigma_{V_0^{rel}-1}$, otherwise (\ref{eq:10}) is trivially true.

Whenever $r\geq 2$, apart from the special case we dealt with at the start of this proof, statement (2) of Lemma~\ref{linear system} implies that
$$\nu^{+rel}>p-|\sigma|_1$$
Lemma~\ref{multiples of p on T but minus} implies that
$$V^{rel}_0=V^\sigma_{\nu^{+rel}}>\dfrac{p-|\sigma|_1}{2}.$$
Now we look at the statement of Lemma~\ref{a-b geq x+1}. When $a=V^{rel}_0$, $b=V^{rel}_0-1$, $x=1$, the first two conditions of Lemma~\ref{a-b geq x+1} are true but the conclusion is false. Therefore, the third condition
$$T^\sigma_{V_0^{rel}}-T^\sigma_{V_0^{rel}-1}\geq 2$$
must be false.
\end{proof}

Note that Corollary~\ref{V diff upperbound} implies that when $0\leq k<p$, we have $V^\sigma_{k+1}=V^\sigma_k\text{ or }V^\sigma_k+1$.

Also, note that Corollary~\ref{V diff upperbound} and Corollary~\ref{V diff lowerbound} together imply that for any positive integer $x$, when $0\leq k<p$, $V^\sigma_{xp+k+1}-V^\sigma_{xp+k}$ must be either $x$ or $x+1$.

To prove equation (\ref{eq:9}), by induction on $k$, it suffices to show that
$$V^\sigma_{xp+k+1}-V^\sigma_{xp+k}=\begin{cases}x+1\text{ if }V^\sigma_{k+1}=V^\sigma_k+1,\\
x\phantom{+11}\text{ if }V^\sigma_{k+1}=V^\sigma_k.
\end{cases}$$

\begin{Lemma} \label{xc}
Let $x$ be a positive integer, $0\leq k<p$. If $V^\sigma_k=V^\sigma_{k+1}$, then
$$V^\sigma_{xp+k+1}-V^\sigma_{xp+k}=x.$$
\end{Lemma}

\begin{proof}
By Corollary~\ref{V diff lowerbound}, we know that $V^\sigma_{xp+k+1}>V^\sigma_{xp+k}$, and therefore $T^\sigma_{V^\sigma_{xp+k}}=xp+k$.

Let $\alpha\in S_{\leq V^\sigma_{xp+k}}$ be such that $\alpha\cdot\sigma=T^\sigma_{V^\sigma_{xp+k}}=xp+k$. Since $V^\sigma_{xp+k}$ is a relevant coefficient, we know that $T^\sigma_{V^\sigma_{xp+k}}>T^\sigma_{V^\sigma_{xp+k}-1}$, and hence $\alpha\not\in S_{\leq V^\sigma_{xp+k}-1}$. So, we know that
$$\dfrac{1}{2}\sum\limits_j\alpha_j(\alpha_j+1)=V^\sigma_{xp+k}.$$
Since $V^\sigma_{xp+k+1}$ is a relevant coefficient, we know that $T^\sigma_{V^\sigma_{xp+k+1}}>T^\sigma_{V^\sigma_{xp+k+1}-1}$. Since there are no relevant coefficients between $V^\sigma_{xp+k+1}$ and $V^\sigma_{xp+k}$, we also know that $T^\sigma_{V^\sigma_{xp+k}}=T^\sigma_{V^\sigma_{xp+k+1}-1}$.

Now we try to apply Lemma~\ref{a-b leq x} with $a=V^\sigma_{xp+k+1}$, $b=V^\sigma_{xp+k}$. We do not have the first condition $a\leq\tfrac{x^2p+x|\sigma|_1}{2}$. However, in the proof of Lemma~\ref{a-b leq x}, the only place where the condition is used, is to show that there exists some $i$ such that $\alpha_i<x\sigma_i$. If such an $i$ exists, even if the condition $a\leq\tfrac{x^2p+x|\sigma|_1}{2}$ is not satisfied, the proof of Lemma~\ref{a-b leq x} still goes through and we can conclude that
$$V^\sigma_{xp+k+1}-V^\sigma_{xp+k}\leq x.$$
Hence, we can assume that such an $i$ does not exist. In that case, $\alpha=\beta+x\sigma$ for some $\beta$ with all entries being non-negative. Note that $\beta$ satisfies
\begin{align*}
\beta\cdot\sigma&=(\alpha-x\sigma)\cdot\sigma\\
&=\alpha\cdot\sigma-xp\\
&=k.
\end{align*}
\indent Notice that $V^\sigma_k-1$ is less than the $k^{th}$ relevant coefficient $V_k$. Therefore, $T^\sigma_{V^\sigma_k-1}<k$, and we can conclude that
$$\dfrac{1}{2}\sum\limits_j\beta_j(\beta_j+1)\geq V^\sigma_k.$$
Now we have
\begin{align*}
V^\sigma_{xp+k}&=\sum\limits_j\alpha_j(\alpha_j+1)\\
&=\sum\limits_j(\beta_j+x\sigma_j)(\beta_j+x\sigma_j+1)\\
&=\dfrac{x^2p+x|\sigma|_1}{2}+x(\beta\cdot\sigma)+\dfrac{1}{2}\sum\limits_j\beta_j(\beta_j+1)\\
&\geq\dfrac{x^2p+x|\sigma|_1}{2}+xk+V^\sigma_k.
\end{align*}
Let $c$ be maximal such that $V^\sigma_k=V^\sigma_{k+c}$.

By the maximality of $c$, $T^\sigma_{V^\sigma_k}=k+c$. Let $\beta'\in S_{\leq V_k}$ be such that $\beta'\cdot\sigma=k+c$.

Let $\alpha':=\beta'+x\sigma$. Then we have $\alpha'\cdot\sigma=xp+k+c$. Also,
\begin{align*}
\sum\limits_j\alpha'_j(\alpha'_j+1)&=\sum\limits_j(\beta'_j+x\sigma_j)(\beta'_j+x\sigma_j+1)\\
&=\dfrac{x^2p+x|\sigma|_1}{2}+x(\beta'\cdot\sigma)+\dfrac{1}{2}\sum\limits_j\beta'_j(\beta'_j+1)\\
&\leq\dfrac{x^2p+x|\sigma|_1}{2}+xk+xc+V^\sigma_k\\
&\leq V^\sigma_{xp+k}+xc.
\end{align*}
Hence, $\alpha'\in S_{V^\sigma_{xp+k}+xc}$, and we can conclude that
$$T^\sigma_{V^\sigma_{xp+k}+xc}\geq xp+k+c.$$
So, there are at least $xp+k+c$ relevant coefficients less than or equal to $V^\sigma_{xp+k}+xc$. Hence, we can conclude that
$$V^\sigma_{xp+k+c}\leq V^\sigma_{xp+k}+xc,$$
which can be arranged into
$$V^\sigma_{xp+k+c}-V^\sigma_{xp+k}\leq xc.$$
By Corollary~\ref{V diff lowerbound}, we know that
\begin{align*}
V^\sigma_{xp+k+1}-V^\sigma_{xp+k}&\geq x,\\
V^\sigma_{xp+k+2}-V^\sigma_{xp+k+1}&\geq x,\\
\dots\\
V^\sigma_{xp+k+c}-V^\sigma_{xp+k+c-1}&\geq x.
\end{align*}
The only way that $V^\sigma_{xp+k+c}-V^\sigma_{xp+k}\leq xc$ can hold is that these inequalities are all equalities.
\end{proof}

Now we prove equation (\ref{eq:9}).

\begin{Lemma} \label{V structure}
Let $x$ be a positive integer and $0\leq k<p$. Then
$$V^{\sigma}_{xp+k}=\dfrac{x^2p+x|\sigma|_1}{2}+xk+V^{\sigma}_k.$$
\end{Lemma}

\begin{proof}
First of all, by Corollary~\ref{multiples of p on V}, we know this is true when $k=0$. Now we consider
\begin{align*}
V^{\sigma}_{xp+p}-V^{\sigma}_{xp}&=\dfrac{(x+1)^2p+(x+1)|\sigma|_1}{2}-\dfrac{x^2p+x|\sigma|_1}{2}\\
&=xp+\dfrac{p+|\sigma|_1}{2}.
\end{align*}
Therefore, by considering Corollary~\ref{V diff upperbound} and Corollary~\ref{V diff lowerbound}, we know that exactly $\tfrac{p+|\sigma|_1}{2}$ values of $k\in\{0,\dots,p-1\}$ satisfy $V^\sigma_{xp+k+1}-V^\sigma_{xp+k}=x+1$, and the rest of the $k$ values satisfies $V^\sigma_{xp+k+1}-V^\sigma_{xp+k}=x$.

Since $V^\sigma_p=\tfrac{p+|\sigma|_1}{2}$, by Corollary~\ref{V diff upperbound} we know that exactly $\tfrac{p+|\sigma|_1}{2}$ values of $k\in\{0,\dots,p-1\}$ satisfy $V^\sigma_{k+1}=V^\sigma_k+1$. Therefore, by considering Lemma~\ref{xc}, we conclude that for $k\in\{0,\dots,p-1\}$,
$$V^\sigma_{xp+k+1}-V^\sigma_{xp+k}=x+1\text{ if and only if }V^\sigma_{k+1}=V^\sigma_k+1$$
and
$$V^\sigma_{xp+k+1}-V^\sigma_{xp+k}=x\text{ if and only if }V^\sigma_{k+1}=V^\sigma_k.$$
The result follows from induction on $k$.
\end{proof}

Before we prove Theorem~\ref{main result 6}, we make a few more evaluations of relevant coefficients.

Let $odd(\sigma)$ be the number of odd entries in $\sigma$.

\begin{Lemma}\label{some values}
The following five statements are true:\\
(1) If $p$ is even, $V^\sigma_{\tfrac{2p-|\sigma|_1}{2}}=\tfrac{p}{2}$.\\
(2) If $p$ is odd, $V^\sigma_{\tfrac{2p-|\sigma|_1-1}{2}}=\tfrac{p-1}{2}$.\\
(3) $V^\sigma_{\tfrac{p-|\sigma|_1}{2}}\geq\tfrac{p-odd(\sigma)}{8}$\\
(4) In (3), equality holds if and only if the even entries of $\sigma$ can be partitioned into two sets with equal sum.\\
(5) $V^\sigma_{\tfrac{p-|\sigma|_1}{2}}\leq\tfrac{p+3odd(\sigma)}{8}$
\end{Lemma}

\begin{proof}
Corollary~\ref{multiples of p on V} and Corollary~\ref{V diff upperbound} and Corollary~\ref{V diff lowerbound} together implies that for any $0\leq k\leq |\sigma|_1$,
$$V^\sigma_{p-k}=\dfrac{p+|\sigma|_1}{2}-k.$$
This implies statements (1) and (2) of Lemma~\ref{some values}.

Now we investigate $V^\sigma_{\tfrac{p-|\sigma|_1}{2}}$.

Consider the full rational greedy algorithm that purchases all options strictly less than $\tfrac{1}{2}$ coin per point. The amount of coins spent is
\begin{align*}
\sum\limits_{\sigma_i \ even}\dfrac{\left(\tfrac{\sigma_i}{2}-1\right)\left(\tfrac{\sigma_i}{2}\right)}{2}-\sum\limits_{\sigma_i \ odd}\dfrac{\left(\tfrac{\sigma_i-1}{2}\right)\left(\tfrac{\sigma_i-1}{2}+1\right)}{2}&=\sum\limits_{\sigma_i \ even}\left(\dfrac{\sigma_i^2}{8}-\dfrac{\sigma_i}{4}\right)+\sum\limits_{\sigma_i \ odd}\left(\dfrac{\sigma_i^2}{8}-\dfrac{1}{8}\right)\\
&=\dfrac{p-odd(\sigma)}{8}-\dfrac{1}{4}\sum\limits_{\sigma_i \ even}\sigma_i.
\end{align*}
The amount of points obtained is
\begin{align*}
\sum\limits_{\sigma_i \ even}\left(\dfrac{\sigma_i}{2}-1\right)\sigma_i+\sum\limits_{\sigma_i \ odd}\left(\dfrac{\sigma_i-1}{2}\right)\sigma_i&=\sum\limits_{\sigma_i \ even}\left(\dfrac{\sigma^2_i}{2}-\sigma_i\right)+\sum\limits_{\sigma_i \ odd}\left(\dfrac{\sigma^2_i}{2}-\dfrac{\sigma_i}{2}\right)\\
&=\dfrac{p-|\sigma|_1}{2}-\dfrac{1}{2}\sum\limits_{\sigma_i \ even}\sigma_i.
\end{align*}
Hence, $T^\sigma_{\tfrac{p-odd(\sigma)}{8}-\tfrac{1}{4}\sum\limits_{\sigma_i \ even}\sigma_i}=\dfrac{p-|\sigma|_1}{2}-\dfrac{1}{2}\sum\limits_{\sigma_i \ even}\sigma_i$.

By also considering the purchase options that are exactly $\tfrac{1}{2}$ coin per point, we conclude that for any $0\leq k\leq\tfrac{1}{2}\sum\limits_{\sigma_i \ even}\sigma_i$,
$$T^{\sigma,\mathbb{Q}}_{\tfrac{p-odd(\sigma)}{8}-\tfrac{1}{4}\sum\limits_{\sigma_i \ even}\sigma_i+k}=\dfrac{p-|\sigma|_1}{2}-\dfrac{1}{2}\sum\limits_{\sigma_i \ even}\sigma_i+\dfrac{k}{2},$$
where $k$ represents the amount of coins spent on purchasing options that exactly $\tfrac{1}{2}$ coin per point under a rational greedy algorithm.

Setting $k=\tfrac{1}{4}\sum\limits_{\sigma_i \ even}\sigma_i$, we conclude that
$$T^{\sigma,\mathbb{Q}}_{\tfrac{p-odd(\sigma)}{8}}=\dfrac{p-|\sigma|_1}{2}.$$
In particular $T^{\sigma,\mathbb{Q}}_{\tfrac{p-odd(\sigma)}{8}}=T^\sigma_{\tfrac{p-odd(\sigma)}{8}}$ if and only if there is an integral implementation of a greedy algorithm that spends exactly $\tfrac{1}{4}\sum\limits_{\sigma_i \ even}\sigma_i$ coins on purchasing options that are exactly $\tfrac{1}{2}$ coin per point. Such an integral implementation is possible if and only if there is a subset of even $\sigma_i$'s that sum to $\tfrac{1}{2}\sum\limits_{\sigma_i \ even}$, which is equivalent to saying that the even entries of $\sigma$ can be partitioned into two sets with equal sum. Therefore,
$$T^\sigma_{\tfrac{p-odd(\sigma)}{8}}\leq\dfrac{p-|\sigma|_1}{2}$$
with equality if and only if the even entries of $\sigma$ can be partitioned into two sets with equal sum.

Note that $T^\sigma_{\tfrac{p-odd(\sigma)}{8}}=\tfrac{p-|\sigma|_1}{2}$ implies $V^\sigma_{\tfrac{p-|\sigma|_1}{2}}\leq\tfrac{p-odd(\sigma)}{8}$ and $V^\sigma_{\tfrac{p-|\sigma|_1}{2}+1}>\tfrac{p-odd(\sigma)}{8}$, which implies $V^\sigma_{\tfrac{p-|\sigma|_1}{2}}=\tfrac{p-odd(\sigma)}{8}$ because of Corollary~\ref{V diff upperbound}. Also, $T^\sigma_{\tfrac{p-odd(\sigma)}{8}}<\tfrac{p-|\sigma|_1}{2}$ implies $V^\sigma_{\tfrac{p-|\sigma|_1}{2}}>T^\sigma_{\tfrac{p-odd(\sigma)}{8}}$. Therefore, this implies statements (3) and (4) of Lemma~\ref{some values}.

Since $\sigma$ is a changemaker vector, there is some $A\subseteq\{0,\dots,n\}$ such that $\sum\limits_{i\in A}\sigma_i=\tfrac{1}{2}\sum\limits_{\sigma_i \ even}\sigma_i$. Consider $\alpha$ defined with
$$\alpha_i=\begin{cases}\tfrac{\sigma_i}{2}-1\text{ if }\sigma_i\text{ is even and }i\not\in A,\\
\tfrac{\sigma_i}{2}\text{ if }\sigma_i\text{ is even and }i\in A,\\
\tfrac{\sigma_i-1}{2}\text{ if }\sigma_i\text{ is odd and }i\not\in A,\\
\tfrac{\sigma_i-1}{2}+1\text{ if }\sigma_i\text{ is odd and }i\in A.
\end{cases}$$
Then
\begin{align*}
\alpha\cdot\sigma&=\dfrac{p-|\sigma|_1}{2}-\dfrac{1}{2}\sum\limits_{\sigma_i \ even}\sigma_i+\sum\limits_{i\in A}\sigma_i\\
&=\dfrac{p-|\sigma|_1}{2}
\end{align*}
and
\begin{align*}
\sum\limits_j\alpha_j(\alpha_j+1)&=\dfrac{p-odd(\sigma)}{8}-\dfrac{1}{4}\sum\limits_{\sigma_i \ even}\sigma_i+\sum_{\substack{\sigma_i \ even \\ i\in A}}\dfrac{\sigma_i}{2}+\sum_{\substack{\sigma_i \ odd \\ i\in A}}\dfrac{\sigma_i+1}{2}\\
&=\dfrac{p-odd(\sigma)}{8}-\dfrac{1}{4}\sum\limits_{\sigma_i \ even}\sigma_i+\dfrac{1}{2}\sum_{i \in A}\sigma_i+\sum_{\substack{\sigma_i \ odd \\ i\in A}}\dfrac{1}{2}\\
&=\dfrac{p-odd(\sigma)}{8}+\dfrac{\text{number of }i\in A\text{ such that }\sigma_i\text{ is odd}}{2}.
\end{align*}
\indent Let $odd_A$ denote the number of $i\in A$ such that $\sigma_i$ is odd. Now we have
$$T^\sigma_{\tfrac{p-odd(\sigma)}{8}+\tfrac{odd_A}{2}}\geq\dfrac{p-|\sigma|_1}{2}.$$
Therefore,
$$V^\sigma_{\tfrac{p-|\sigma|_1}{2}}\leq\tfrac{p+4odd_A-odd(\sigma)}{8}.$$
Since $odd_A\leq odd(\sigma)$, we conclude that
$$V^\sigma_{\tfrac{p-|\sigma|_1}{2}}\leq\tfrac{p+3odd(\sigma)}{8}.$$
\end{proof}

Now we can prove Theorem~\ref{main result 6}.

\begin{proof}[Proof or Theorem~\ref{main result 6}]
When $r,p$ are both even, let $r=2x$.

\begin{align*}
\nu^{+rel}&=\dfrac{2xp-|\sigma|_1}{2}\\
&=(x-1)p+\dfrac{2p-|\sigma|_1}{2}
\end{align*}

\begin{align*}
V_0&=V^{rel}_0\\
&=V^\sigma_{(x-1)p+\tfrac{2p-|\sigma|_1}{2}}\\
&=\dfrac{(x-1)^2p+(x-1)|\sigma|_1}{2}+(x-1)\dfrac{2p-|\sigma|_1}{2}+V^\sigma_{\tfrac{2p-|\sigma|_1}{2}}\\
&=\dfrac{\left(\tfrac{r}{2}-1\right)^2p+\left(\tfrac{r}{2}-1\right)|\sigma|_1}{2}+\left(\dfrac{r}{2}-1\right)\dfrac{2p-|\sigma|_1}{2}+\dfrac{p}{2}
\end{align*}
This simplifies to $V_0=\dfrac{r^2p}{8}$. Hence, $r^2p=8V_0$.

When $r$ is even and $p$ is odd, let $r=2x$.

\begin{align*}
\nu^{+rel}&=-\dfrac{1}{2}+\dfrac{2xp-|\sigma|_1}{2}\\
&=(x-1)p+\dfrac{2p-|\sigma|_1-1}{2}
\end{align*}

\begin{align*}
V_{\tfrac{r}{2}}&=V^{rel}_0\\
&=V^\sigma_{(x-1)p+\tfrac{2p-|\sigma|_1-1}{2}}\\
&=\dfrac{(x-1)^2p+(x-1)|\sigma|_1}{2}+(x-1)\dfrac{2p-|\sigma|_1-1}{2}+V^\sigma_{\tfrac{2p-|\sigma|_1-1}{2}}\\
&=\dfrac{\left(\tfrac{r}{2}-1\right)^2p+\left(\tfrac{r}{2}-1\right)|\sigma|_1}{2}+\left(\dfrac{r}{2}-1\right)\dfrac{2p-|\sigma|_1-1}{2}+\dfrac{p-1}{2}
\end{align*}
This simplifies to $V_{\tfrac{r}{2}}=\dfrac{r^2p}{8}-\dfrac{r}{4}$.

Hence, $V_0\geq V_{\tfrac{r}{2}}=\dfrac{r^2p}{8}-\dfrac{r}{4}$ and $V_0\leq V_{\tfrac{r}{2}}+\tfrac{r}{2}=\dfrac{r^2p}{8}+\dfrac{r}{4}$. This implies
$$8V_0-2r\leq r^2p\leq 8V_0+2r.$$
When $r$ is odd, let $r=2x+1$.

\begin{align*}
\nu^{+rel}&=\dfrac{(2x+1)p-|\sigma|_1}{2}\\
&=xp+\dfrac{p-|\sigma|_1}{2}
\end{align*}

\begin{align*}
V_0&=V^{rel}_0\\
&=V^\sigma_{xp+\tfrac{p-|\sigma|_1}{2}}\\
&=\dfrac{(x^2p+x|\sigma|_1}{2}+x\dfrac{p-|\sigma|_1}{2}+V^\sigma_{\tfrac{p-|\sigma|_1}{2}}\\
&=\dfrac{\left(\tfrac{r-1}{2}\right)^2p+\left(\tfrac{r-1}{2}\right)|\sigma|_1}{2}+\left(\dfrac{r-1}{2}\right)\dfrac{2p-|\sigma|_1}{2}+V^\sigma_{\tfrac{p-|\sigma|_1}{2}}
\end{align*}
This simplifies to $V_0=\tfrac{r^2p-p}{8}+V^\sigma_{\tfrac{p-|\sigma|_1}{2}}$. Hence,
$$\dfrac{r^2p-p}{8}+\dfrac{p-odd(\sigma)}{8}\leq V_0\leq \dfrac{r^2p-p}{8}+\dfrac{p+3odd(\sigma)}{8},$$
and $\tfrac{r^2p-p}{8}+\tfrac{p-odd(\sigma)}{8}=V_0$ if and only if the even entries of $\sigma$ can be partitioned into two sets with equal sum. The result follows from rearranging this inequality.
\end{proof}

\section{Linear changemakers}

Greene \cite{realization} showed that $\sigma$ must take 1 of the 33 forms, including 26 small families which have a simple structure and 7 large families that involves a process called ``weight expansion''. In this section, we work on 1 of the 26 small families. We believe that a similar approach can be used on all small families. Unfortunately, we do not see how does this generalize to the large families.

The small family structure that we investigate in the section has changemakers in the form

$$\sigma=\begin{pmatrix}
    4s+3,&2s+1,&s+1,&s,&\smash{\underbrace{\begin{matrix}
        1,&\dots&,&1
    \end{matrix}}_{s\text{ copies of }1\text{'s}}}
\end{pmatrix},$$

where $s\geq 2$.

According to Greene's work \cite{realization}, this family corresponds to the $j\leq -3$ case in Berge Type X in Rasmussen's table \cite[\S 6.2]{rasmussen2007}.

At the end of this section we prove Theorem~\ref{greene type 1}, which states that any knot in $S^3$ has at most 7 lensbordant surgeries with an associated changemaker coming from this family.

First we use the following lemma to convert that into a statement about the number of possible $\sigma$ instead of the number of possible surgeries.

\begin{Lemma} \label{at most one r}
Given a knot $K$ in $S^3$ and a changemaker vector $\sigma$, if $\sigma\neq(1)$ or $\nu^+(K)$ is not a triangular number, then there is at most one lensbordant surgery on $K$ associated with $\sigma$. When $\sigma=(1)$ and $\nu^+(K)$ is a triangular number, there are at most two lensbordant surgeries on $K$ associated with $\sigma$.
\end{Lemma}

\begin{Rem}
When $\sigma=(1)$, being a lensbordant surgery associated with $\sigma$ means that the surgery bounds a rational homology ball.
\end{Rem}

\begin{proof}[Proof of Lemma~\ref{at most one r}]
Suppose if the $r_1^2p$-surgery and the $r_2^2p$-surgery on $K$ are both lensbordant associated with $\sigma$. Let $r_2\geq r_1$.

Theorem~\ref{main result 4} states that
$$2\nu^++2(r-1)\geq r^2p-r|\sigma|_1\geq 2\nu^+.$$
Hence,
$$r_1^2p-r_1|\sigma|_1\geq r_2^2p-r_2|\sigma|_1-2(r_2-1),$$
which can be rearranged into
$$0\geq(r_2-r_1)\left((r_2-r_1)p-|\sigma|_1\right)+(2r_1p-2)(r_2-r_1-1)+2r_1(p-1).$$
Therefore, either $r_2=r_1$ or ($r_2=r_1+1$ and $p=1$).

When $p=1$, we know that
$$2\nu^++2(r-1)\geq r^2-r\geq 2\nu^+,$$
which can be rearranged into
$$\dfrac{r(r-1)}{2}\geq\nu^+\geq\dfrac{(r-1)(r-2)}{2}.$$
Hence, we can determine the value of $r$ unless $\nu^+$ is a triangular number, in which case there can be two values of $r$.
\end{proof}

\begin{Rem}
The case with two values of $r$ can indeed happen. Both the $p^2$-surgery and the $(p+1)^2$-surgery on the torus knot $T_{p,(p+1)}$ bound a smooth rational homology ball \cite[Th. 1.4]{aceto2017}.
\end{Rem}

Given Lemma~\ref{at most one r}, we now focus on the number of $\sigma$'s instead of the number of surgeries.

The majority of the content in this section is for proving the following:

\begin{Th}\label{greene type 1 2}
Let $s\geq 5$ and $r\geq 2$. Let $K$ be a knot in $S^3$ and there is a $(r,p)$-lensbordant surgery on $K$ associated with a changemaker vector in the form

$$\sigma=\begin{pmatrix}
    4s+3,&2s+1,&s+1,&s,&\smash{\underbrace{\begin{matrix}
        1,&\dots&,&1
    \end{matrix}}_{s\text{ copies of }1\text{'s}}}
\end{pmatrix}.$$

\bigskip

\bigskip

Let $a$ be the number of elements in the set $\{i\geq 0 \mid 296\leq V_i(K)\leq 300\}$.

Let $b$ be the biggest value satisfying the property that there are at least $\dfrac{a}{2}$ elements in the set $\{i\geq 0 \mid V_i(K)=b\}$.

Then
$$s=\left\lfloor\sqrt{\dfrac{b}{11}}\right\rfloor\text{ or }\left\lfloor\sqrt{\dfrac{b}{11}}\right\rfloor+1\text{ or }\left\lfloor\sqrt{\dfrac{b}{11}}\right\rfloor+2.$$
\end{Th}

From this point onwards, until we finish proving Theorem~\ref{greene type 1 2}, we let

$$\sigma=\begin{pmatrix}
    4s+3,&2s+1,&s+1,&s,&\smash{\underbrace{\begin{matrix}
        1,&\dots&,&1
    \end{matrix}}_{s\text{ copies of }1\text{'s}}}
\end{pmatrix}$$

\bigskip

and we assume that $s\geq 5$ and $r\geq 2$.

The following lemma converts information about $V^{rel}_i$'s into information about $V_i$'s.

\begin{Lemma}\label{conversion}
For any $m_1>m_2\geq 0$,
$$(T^{rel}_{m_1}-T^{rel}_{m_2}+1)r>|\{i\geq 0 \mid m_1\geq V_i>m_2\}|>(T^{rel}_{m_1}-T^{rel}_{m_2}-1)r.$$
\end{Lemma}

\begin{proof}
Notice that $T^{rel}_{m_1}-T^{rel}_{m_2}$ is the number of relevant coefficients having values that are both $\leq m_1$ and $>m_2$. Let $T$ be maximal such that $V_T$ is relevant and $>m_2$. Let $t$ be minimal such that $V_t$ is relevant and $\leq m_1$. Since consecutive relevant coefficients have indices differing by $r$, we have
$$T-t=(T^{rel}_{m_1}-T^{rel}_{m_2}-1)r.$$
Hence, by monotonicity of the $V$ coefficients, we conclude that
$$|\{i\geq 0 \mid m_1\geq V_i>m_2\}|>(T^{rel}_{m_1}-T^{rel}_{m_2}-1)r.$$
Also, since $V_{T+r}\leq m_2$, and also if $t\geq r$ then $V_{t-r}>m_1$, we have
\begin{align*}
|\{i\geq 0 \mid m_1\geq V_i>m_2\}|&\leq(T+r-1)-(t-r+1)+1\\
&=T-t+2r-1\\
&=(T^{rel}_{m_1}-T^{rel}_{m_2}+1)r-1.
\end{align*}
\end{proof}

Now we calculate some values of $T^\sigma$.

\begin{Lemma}
$$T^\sigma_{295}=110s+75$$
\end{Lemma}

\begin{proof}
With 295 coins, when $s\geq 15$, an integral rational greedy algorithm up to $\tfrac{5}{s}$ coin per point gives
$$\alpha=(20,10,5,5,0,\dots),$$
which implies
$$T^\sigma_{295}=110s+75.$$
\indent When $14\geq s\geq 5$, rational greedy algorithm up to $\tfrac{21}{4s+3}$ coin per point gives
$$\alpha=(20+\dfrac{5}{21},10,5,4,0,\dots),$$
which implies
$$T^{\sigma,\mathbb{Q}}_{295}=110s+75+\dfrac{15-s}{21}<110s+76.$$
Hence,
$$T^\sigma_{295}\leq 110s+75.$$
By considering $\alpha=(20,10,5,5,0,\dots)$, we conclude that
$$T^\sigma_{295}=110s+75.$$
\end{proof}

\begin{Lemma}
$$T^\sigma_{300}=110s+80$$
\end{Lemma}

\begin{proof}
By considering $\alpha=(20,10,5,5,1,1,1,1,1,0,\dots)$, we know that
$$T^\sigma_{300}\geq 110s+80.$$
Hence, it suffices to show that
$$T^\sigma_{300}\leq 110s+80.$$
\indent Let $m=\tfrac{1}{2}\sum\limits_j \alpha_j(\alpha_j+1)$. Let $\alpha=(\alpha_0,\dots,\alpha_n)$. Then
$$\alpha\cdot\sigma=\alpha_0(4s+3)+\alpha_1(2s+1)+\alpha_2(s+1)+\alpha_3s+(\alpha_4+\dots+\alpha_n).$$
Let
$$S=\{\alpha\in S_{\leq 300} \mid \alpha\cdot\sigma=T^\sigma_{300}\text{ and }\alpha_3\geq 5\}.$$
First of all, we argue that if the set $S$ is non-empty, then $T^\sigma_{300}=110s+80$.

If $S$ is non-empty, pick an $\alpha$ in $S$ with minimal $\alpha_3$. We know that $\alpha_2\geq 5$ because otherwise we can swap $\alpha_2,\alpha_3$ to increase $\alpha\cdot\sigma$ without increasing $m$. So, $\alpha_1\geq 10$ because otherwise we can decrease $\alpha_2,\alpha_3$ by 1 and increase $\alpha_1$ by 1, which does not increase $m$ and keep $\alpha\cdot\sigma$ unchanged while decreasing $\alpha_3$. So, $\alpha_0\geq 20$ because otherwise we can decrease $\alpha_1,\alpha_2,\alpha_3$ by 1 and increase $\alpha_0$ by 1 to increase $\alpha\cdot\sigma$ without increasing $m$. Notice that
$$\dfrac{20\times 21}{2}+\dfrac{10\times 11}{2}+\dfrac{5\times 6}{2}+\dfrac{5\times 6}{2}=295.$$
Therefore $\alpha$ must be $(20,10,5,5,1,1,1,1,1)$ because otherwise $m$ will be bigger than 300. This gives
$$T^\sigma_{300}=110s+80.$$
\indent Now we consider the case when $S$ is empty. In this case, all optimal $\alpha$'s have $\alpha_3\leq 4$. Among these $\alpha$'s, pick one with maximal $\alpha_3$ and satisfies
$$\alpha_4\geq\alpha_5\geq\dots .$$
We know that at most $4$ values of $\alpha_4,\alpha_5,\dots$ can be non-zero, because otherwise we can decrease $(\alpha_3+1)$ of those values by 1 and increase $\alpha_3$ by 1, which increases $\alpha_3$ but does not increase $m$ and does not decrease $\alpha\cdot\sigma$.

Let $\sigma_1=(4,2,1,1,0,\dots)$ and $\sigma_2=(19,9,5,4,1,1,1,1,0,\dots)$. We have
$$\alpha\cdot\sigma=(s-4)\alpha\cdot\sigma_1+\alpha\cdot\sigma_2.$$
\indent First we investigate $\sigma_1$. It has $p_1=22$ and $|\sigma_1|_1=8$. Lemma~\ref{V structure} implies
$$V^{\sigma_1}_{110}=295\text{ and }V^{\sigma_1}_{111}=301.$$
Therefore, $T^{\sigma_1}_{300}=110$.

Now we investigate $\sigma_2$. It has $p_2=487$ and $|\sigma_2|_1=41$. Since $T^{\sigma_2}_2=28$ and $T^{\sigma_2}_3=38$, we know that $V^{\sigma_2}_{33}=V^{\sigma_2}_{34}=3$. Lemma~\ref{V structure} implies
$$V^{\sigma_2}_{520}=300\text{ and }V^{\sigma_1}_{521}=301.$$
Therefore, $T^{\sigma_2}_{300}=520$.

Hence,
\begin{align*}
\alpha\cdot\sigma&=(s-4)\alpha\cdot\sigma_1+\alpha\cdot\sigma_2\\
&\leq 110(s-4)+520\\
&=110s+80.
\end{align*}
\end{proof}

\begin{Lemma}
$$T^\sigma_{11s^2-33s+24}\leq 22s^2-22s-28$$
\end{Lemma}

\begin{proof}
Let $\alpha\in S_{\leq 11s^2-33s+24}$ be such that $\alpha\cdot\sigma=T^\sigma_{11s^2-33s+24}$.

Let $\sigma_1=(4,2,1,1,0,\dots)$ and $\sigma_2=(4s-5,2s-3,s-1,s-2,1,\dots,1)$. We have
$$\alpha\cdot\sigma=2\alpha\cdot\sigma_1+\alpha\cdot\sigma_2.$$
First we investigate $\sigma_1$. It has $p_1=22$ and $|\sigma_1|_1=8$. Since $T^{\sigma_1}_1=4$ and $T^{\sigma_1}_2=6$ and $T^{\sigma_1}_3=8$, we know that $V^{\sigma_1}_6=2$ and $V^{\sigma_1}_7=3$. Lemma~\ref{V structure} implies
$$V^{\sigma_1}_{22(s-2)+6}=11s^2-34s+26\text{ and }V^{\sigma_1}_{22(s-2)+7}=11s^2-33s+25.$$
Therefore,
$$T^{\sigma_1}_{11s^2-33s+24}=22(s-2)+6=22s-38.$$
\indent Now we investigate $\sigma_2$. Rational greedy algorithm up to $\tfrac{4s-6}{4s-5}$ coin per point gives
$$\alpha=(4s-7+\dfrac{4s-7}{4s-6},2s-4,s-2,s-3,0,\dots),$$
which implies
$$T^{\sigma_2,\mathbb{Q}}_{11s^2-33s+24}=22s^2-66s+48+\dfrac{4s-7}{4s-6},$$
which implies
$$T^{\sigma_2}_{11s^2-33s+24}\leq 22s^2-66s+48.$$
Hence,
\begin{align*}
\alpha\cdot\sigma&=2\alpha\cdot\sigma_1+\alpha\cdot\sigma_2\\
&\leq 2(22s-38)+22s^2-66s+48\\
&=22s^2-22s-28.
\end{align*}
\end{proof}

\begin{Lemma}
$$T^\sigma_{11s^2-33s+25}\geq 22s^2-22s-24$$
\end{Lemma}

\begin{proof}
This follows from observing that when $\alpha=(4s-6,2s-4,s-2,s-3)$, we have
$$\alpha\in S_{11s^2-33s+25}\text{ and }\alpha\cdot\sigma=22s^2-22s-24.$$
\end{proof}

Now we can prove Theorem~\ref{greene type 1 2}, which states that when $s\geq 5$ and $r\geq 2$, if we let $a$ be the number of elements in the set $\{i\geq 0 \mid 296\leq V_i\leq 300\}$ and let $b$ be the biggest value satisfying the property that there are at least $\dfrac{a}{2}$ elements in the set $\{i\geq 0 \mid V_i=b\}$, then
$$s=\left\lfloor\sqrt{\dfrac{b}{11}}\right\rfloor\text{ or }\left\lfloor\sqrt{\dfrac{b}{11}}\right\rfloor+1\text{ or }\left\lfloor\sqrt{\dfrac{b}{11}}\right\rfloor+2.$$

\begin{proof}[Proof of Theorem~\ref{greene type 1 2}]
First of all, note that $p=22s^2+31s+11$ and $|\sigma|_1=9s+5$. So we have
$$T^{rel}_{V_0^{rel}}=\nu^{+rel}\geq -\dfrac{1}{2}+\dfrac{1}{2}(2p-|\sigma|_1)=22s^2+\dfrac{53}{2}s+8.$$
Hence, when $r\geq 2$,
$$T^\sigma_{V_0^{rel}}\geq 22s^2+\dfrac{53}{2}s+8.$$
Notice that when $s\geq 5$, $22s^2+\dfrac{53}{2}s+8$ is strictly bigger than $110s+80$ and $22s^2-22s-28$.  Lemma 7.5 and Lemma 7.6 implies that $300$ and $11s^2-33s+24$ are strictly less than $V_0^{rel}$. This implies $295$, $300$, $11s^2-33s+24$, and $11s^2-33s+25$ are all less than or equal to $V_0^{rel}$. Hence, in the statements of Lemma 7.4-7.7, $T^\sigma$ can be replaced by $T^{rel}$.

Together with Lemma~\ref{conversion}, we conclude that 
$$6r>a>4r$$
and
$$|\{i\geq 0 \mid V_i={11s^2-33s+25}\}|>3r.$$
Hence, we know that $b\geq 11s^2-33s+25$.

To get an upper bound of $b$ in terms of $s$, notice that since 
$$|\{i\geq 0 \mid V_i=b\}|\geq\dfrac{a}{2}>2r,$$
Lemma~\ref{conversion} implies that $T^{rel}_b-T^{rel}_{b-1}\geq 2$. Lemma~\ref{a-b geq x+1} implies that
$$b\leq\dfrac{p-|\sigma|_1}{2}=11s^2+11s+3.$$
The result follows from
$$11(s-2)^2<11s^2-33s+25\leq b\leq 11s^2+11s+3<11(s+1)^2.$$
\end{proof}

We conclude this section by proving Theorem~\ref{greene type 1}, which states that any knot in $S^3$ has at most 7 lensbordant surgeries with an associated changemaker coming from this family.

\begin{proof}[Proof of Theorem~\ref{greene type 1}]
Fix a knot in $S^3$.

When $r=1$, $T^{rel}_1=4s+3$ is the number of $V_i$'s with value 1. Therefore, there is at most 1 possible value of $s$.

When $r\geq 2$, Theorem~\ref{greene type 1 2} implies that there are at most 3 possible values of $s\geq 5$.

Hence, together with the possibilities $s=2,3,4$, there are at most 7 possible values of $s$. The result follows from Lemma~\ref{at most one r}.
\end{proof}

\FloatBarrier

\section{Knots in the Poincar\'e homology sphere}

In this section, we work on knots in the Poincar\'e homology sphere (denoted as $\mathcal{P}$) instead of in $S^3$. We use Caudell's work \cite{caudell} to generalize our work in Section 3 of this paper to knots in $\mathcal{P}$.

We make the following definitions that can be found in \cite[\S 1.4]{caudell}:

Let $a\leq b$ be integers that are either both odd or both even. We define $PI(a,b)$ to be the set of integers between $a$ and $b$ (including $a,b$) that have the same parity (odd or even) as $a,b$.

For every negative-definite unimodular lattice $L$,
$$m(L):=\max\{\langle\mathfrak{c},\mathfrak{c}\rangle \mid \mathfrak{c}\in Char(L)\},$$

$$short(L):=\{\mathfrak{c} \mid \mathfrak{c}\in Char(L)\text{ and }\langle\mathfrak{c},\mathfrak{c}\rangle=m(L)\},$$

$$Short(L):=\{\mathfrak{c} \mid \mathfrak{c}\in Char(L)\text{ and }\langle\mathfrak{c},\mathfrak{c}\rangle=m(L)-8\}.$$
For every vector $v\in L$,
$$c(v):=\max\{\langle\mathfrak{c},v\rangle \mid \mathfrak{c}\in short(L)\},$$

$$C(v):=\max\{\langle\mathfrak{c},v\rangle \mid \mathfrak{c}\in Short(L)\}.$$

\begin{Def}\cite[Def 1.6]{caudell} \label{e8}
A vector $\tau=(s,\sigma)\in -E_8\oplus-\mathbb{Z}^{n-7}$ is an $E_8$-changemaker if 
$$PI(-c(\tau),c(\tau))=\{\langle\mathfrak{c},\tau\rangle \mid \mathfrak{c}\in short(-E_8\oplus-\mathbb{Z}^{n-7})\}$$
and
$$PI(c(\tau)+2,C(\tau))\subset\{\langle\mathfrak{c},\tau\rangle \mid \mathfrak{c}\in Short(-E_8\oplus-\mathbb{Z}^{n-7})\}.$$
\end{Def}

Caudell noted that $short(-E_8\oplus-\mathbb{Z}^{n-7})=\{0\}\oplus\{\pm 1\}^{n-7}$, and therefore the first condition of Definition~\ref{e8} is equivalent to saying that $\sigma$ satisfies the changemaker condition (as in Definition~\ref{changemaker def}, but the entries of $\sigma$ are allowed to be 0).

Note that the definition of changemaker in \cite{caudell} allows the entries to be 0. That is different from Definition~\ref{changemaker def} which comes from \cite{realization}.

One important result by Caudell \cite[Th. 1.19]{caudell} is that if the canonical negative-definite plumbing 4-manifold of $L(p,q)$ embeds in $-E_8\oplus-\mathbb{Z}^{n-7}$ as an orthogonal complement of an $E_8$-changemaker, then $L(p,q)$ is a positive integer surgery on one of the knots described by Tange \cite{tange} in $\mathcal{P}$. In this section, we use this to generalize the content of Section 3 to L-space knots in $\mathcal{P}$.

The aim of this section is to prove:

\begin{theorem*}[\ref{poincare knot}]
Suppose the $r^2p$-surgery on a knot $K\subset\mathcal{P}$ produces an L-space that is smoothly $\mathbb{Z}_2$-homology cobordant to a reduced $L(p,q)$. Let
$$i(K,r,p)=\begin{cases}2g(K)+r\text{ if }p\text{ is odd,} \\ 2g(K) \quad \ \ \text{ if }p\text{ is even.}\end{cases}$$
If $r^2p\geq i(K,r,p)$, then $L(p,q)$ must be a positive integer surgery on a knot in $\mathcal{P}$.

If $r^2p<i(K,r,p)$, then $L(p,q)$ must be a positive integer surgery on a knot in $S^3$.
\end{theorem*}

We use a similar set up as Section 3. In this section, we let $K$ be an L-space knot in $K\subset\mathcal{P}$. Let $L(p,q)$ be a reduced lens space that is smoothly $Z_2$-homology cobordant to the $r^2p$-surgery on $K$ which is denoted as $K_{r^2p}$. Let $W$ be the cobordism. Let $P$ be the canonical negative-definite plumbed 4-manifold with boundary $L(p,q)$. Let $W_{r^2p}$ be the 4-manifold obtained by attaching a $r^2p$-framed 2-handle along $K\subset\mathcal{P}\subset\partial(\mathcal{P}\times[0,1])$. Let $i(K,r,p)$ be as described in Theorem~\ref{poincare knot}. Note that we use $\mathcal{P}$ for the Poincar\'e homology sphere and $P$ for the plumbing, following the notation of \cite{caudell} and \cite{ACP}, respectively.

Since $L(p,q)$ is smoothly $Z_2$-homology cobordant to $K_{r^2p}$, it must be smoothly rational homology cobordant to $K_{r^2p}$, and we also know that $|H_1(W)|$ is odd.

We also define the surface $\Sigma$ and label the $\text{spin}^\text{c}$ structures on $K_{r^2p}$ in the same way as in Section 3.

Let $X=P\cup_{L(p,q)}W$ and $Z=X\cup_{K_{r^2p}}-W_{r^2p}$.

\begin{figure}
\centering
\includegraphics[width=0.5\textwidth]{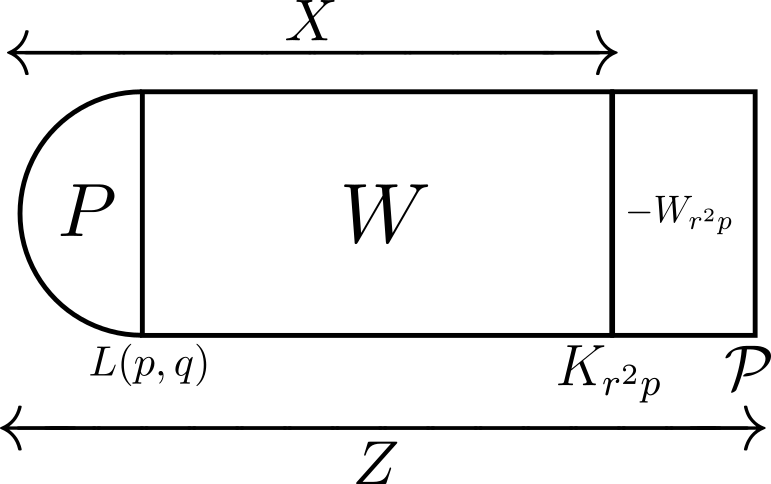}
\caption{The set up of Section 8.}
\end{figure}

Since $\mathcal{P}$ is an integer homology 3-sphere, Lemma 3.2-3.8 and Definition 3.9 still applies. To use notation similar to \cite{caudell}, we let $\tau$ (instead of $\sigma$) to be a generator of the free part of $H_2(W\cup-W_{r^2p})$ such that
$$[\Sigma]=r\tau+torsion.$$
\indent In Section 3, the first time we used the information that $K\subset S^3$ was in the discussion between Lemma 3.8 and Definition 3.9 when we introduced the $V$ coefficients and how they relate the $d$-invariants. So, we start again from there.

Instead of using those $V_i$ coefficients, we use the torsion coefficients $t_i$ defined as follows. Consider the symmetrized Alexander polynomial of $K$
$$\Delta_K(T)=\sum\limits_ja_jT^j.$$
The torsion coefficients are defined as
$$t_i(K):=\sum\limits_{j\geq 1}ja_{|i|+j}.$$
Some properties of torsion coefficients of L-space knots in $\mathcal{P}$ addressed in \cite{tange} and also mentioned in \cite{caudell2021lens}. For $i\geq 0$, the $t_i$ coefficients form a monotonic decreasing sequence of non-negative integers. Also, for any positive integer $k$, whenever $|i|\leq\tfrac{k}{2}$,
\begin{equation}\label{eq:11}
2-2t_i(K)=d(K_k,i)-d(U_k,i),
\end{equation}
where $K_k$ is the $k$-surgery on $K\subset\mathcal{P}$, and $U_k$ is the $k$-surgery on the unknot in $S^3$ (which is $L(k,k-1)$). The $i$ behind denotes the $\text{spin}^\text{c}$ structure with label $i$.

Equation (\ref{eq:11}) is almost the same as what we saw in Section 3, but with an extra 2 on the left hand side coming from the fact that $d(\mathcal{P})=2$.

\begin{Lemma} \label{new 8V}
Let $i$ be relevant. Then
$$8-8t_i=\max_{\substack{\mathfrak{c}\in Char(Z) \\ \langle\mathfrak{c},[\Sigma]\rangle+r^2p=2i}}(\mathfrak{c}^2+(n+1)).$$
\end{Lemma}

\begin{proof}
Everything in the proof of Lemma~\ref{8V} still applies, with $t_i$ replacing $V_i$, and an extra 8 on the left hand side coming from the extra 2 on the left hand side of equation (\ref{eq:11}).
\end{proof}

Recall that in this section, $|H_1(W)|$ is odd. Consider the Mayer-Vietoris sequence
$$\rightarrow H_1(P\cup-W_{r^2p})\oplus H_1(W)\rightarrow H_1(Z)\rightarrow H_0(L(p,q)\cup K_{r^2p})\xrightarrow{injective}.$$
Since $H_1(P\cup-W_{r^2p})=0$, we know that $|H_1(Z)|$ divides $|H_1(W)|$. Hence, $H_1(Z)$ has no 2-torsion.

\begin{Lemma}\label{which one is which}
If $r^2p\geq i(K,r,p)$, the intersection form of $Z$ is $-E_8\oplus-\mathbb{Z}^{n-7}$.\\
\\
If $r^2p<i(K,r,p)$, the intersection form of $Z$ is $-\mathbb{Z}^{n+1}$.
\end{Lemma}

\begin{proof}
According to Scaduto's work \cite[Cor. 1.4]{scaduto}, if a smooth compact oriented negative-definite 4-manifold with $\mathcal{P}$ boundary and $b_2=n+1$ has no 2-torsion in its homology, then its intersection form is either $-\mathbb{Z}^{n+1}$ or $-E_8\oplus-\mathbb{Z}^{n-7}$.

Since $K$ is an L-space knot in $\mathcal{P}$, by combining \cite[Prop. 3.7]{Wu}, \cite[Prop. 3.5]{Ni}, and \cite[Prop. 3.1]{rasmussen2007}, we know that the genus $g(K)$ coincides with the leading degree of the symmetrized Alexander polynomial.

By comparing the definition of $i(K,r,p)$ and Definition~\ref{relevant}, and consider the fact that relevant coefficients are monotonic decreasing, we see that $r^2p\geq i(K,r,p)$ if and only if there exists some relevant $i$ such that $t_i=0$.

Hence, it suffices to show that when there exists some relevant $i$ such that $t_i=0$, the intersection form of $Z$ cannot be $-\mathbb{Z}^{n+1}$, and show that when the intersection form of $Z$ is $-E_8\oplus-\mathbb{Z}^{n-7}$, there exists some relevant $i$ such that $t_i=0$.

Suppose there exists some relevant $i$ such that $t_i=0$.

Similar to the argument in \cite[\S 3]{caudell}, observe that for any characteristic vector $\mathfrak{c}$ in $-\mathbb{Z}^{n+1}$, $\mathfrak{c}^2$ is at most $-(n+1)$. Hence, it is impossible to obtain
$$8=\mathfrak{c}^2+(n+1).$$
Therefore, we can conclude that the intersection form of $Z$ cannot be $-\mathbb{Z}^{n+1}$.

Suppose the intersection form of $Z$ is $-E_8\oplus-\mathbb{Z}^{n-7}$.

Let $\tau=s+\sigma$, with $s\in-E_8$ and $\sigma\in-\mathbb{Z}^{n+1}$. By choosing a suitable orthonormal basis, we can assume that all entries of $\sigma$ are non-negative.

By considering $\mathfrak{c}=(0,(1,\dots,1))\in short(-E_8\oplus-\mathbb{Z}^{n-7})$, Lemma~\ref{new 8V} implies that
$$t_{\tfrac{1}{2}r(rp-|\sigma|_1)}=0.$$
Note that
$$p=\langle\tau,\tau\rangle=-\langle s,s\rangle-\langle\sigma,\sigma\rangle.$$
Since $E_8$ is an even lattice, $\langle s,s\rangle$ is always even. Hence, $p$ has the same parity (odd or even) as $|\sigma|_1$. Therefore, $t_{\tfrac{1}{2}r(rp-|\sigma|_1)}$ is relevant.
\end{proof}

\begin{Lemma}\label{intersection form case 1}
If the intersection form of $Z$ is $-\mathbb{Z}^{n+1}$, then $L(p,q)$ is a positive integer surgery on a knot in $S^3$.
\end{Lemma}

\begin{proof}
This is implied by the argument in Section 3, with $t_i$ replacing $V_i$, $\tau$ replacing $\sigma$, Lemma~\ref{new 8V} replacing Lemma~\ref{8V}, and in the proof of Lemma~\ref{is changemaker} we consider the relevant coefficients that are equal to 1 instead of relevant coefficients that are equal to 0.
\end{proof}

\textbf{From this point onwards, until we start proving Theorem~\ref{poincare knot}, we assume that the intersection form of $Z$ is $-E_8\oplus-\mathbb{Z}^{n-7}$.}

Let $\tau=s+\sigma$, with $s\in-E_8$ and $\sigma\in-\mathbb{Z}^{n-7}$. By choosing a suitable orthonormal basis, we can assume that all entries of $\sigma$ are non-negative.\\
\\
\indent We now prove that $\tau$ is an $E_8$-changemaker. The definition of $E_8$-changemaker consists of two conditions. We prove them one by one using arguments similar to the ones in \cite[\S 3]{caudell}.

\begin{Lemma}
$\sigma$ is a changemaker (with entries allowed to be 0).
\end{Lemma}

\begin{proof}
We follow a similar argument as in the proof of Lemma~\ref{is changemaker}.

When $\mathfrak{c}=(0,(1,\dots,1))\in short(-E_8\oplus-\mathbb{Z}^{n-7})$, we have $\mathfrak{c}^2+(n+1)=8$ and
$$\langle\mathfrak{c},[\Sigma]\rangle+r^2p=2\left(\dfrac{r^2p-r|\sigma|_1}{2}\right).$$
By Lemma~\ref{new 8V}, we conclude that
$$t_{\tfrac{1}{2}r(rp-|\sigma|_1)}=0.$$
By Lemma~\ref{new 8V}, we know that for every relevant $i\geq\tfrac{1}{2}r(rp-|\sigma|_1)$, there exists some $\mathfrak{c}\in Char(-E_8\oplus-\mathbb{Z}^{n-7})$ such that $\mathfrak{c}^2=-(n-7)$ and $\langle\mathfrak{c},[\Sigma]\rangle+r^2p=2i$. Since $short(-E_8\oplus-\mathbb{Z}^{n-7})=\{0\}\oplus\{\pm 1\}^{n-7}$, we know that such a $\mathfrak{c}$ must be in $\{0\}\oplus\{\pm 1\}^{n-7}$. The rest of the proof of Lemma~\ref{is changemaker} applies.
\end{proof}

\begin{Lemma}\label{-4}
If $\langle s,s\rangle\leq-4$, then
$$PI(c(\tau)+2,C(\tau))\subset\{\langle\mathfrak{c},\tau\rangle \mid \mathfrak{c}\in Short(-E_8\oplus-\mathbb{Z}^{n-7})\}.$$
\end{Lemma}

\begin{proof}
From \cite[Lemma 3.4]{caudell}, we know that if $\langle s,s\rangle\leq-4$, then for all $\mathfrak{c}\in Short(-E_8\oplus-\mathbb{Z}^{n-7})$, we have $|\langle\mathfrak{c},\tau\rangle|\leq|\langle\tau,\tau\rangle|=p$. Hence, for all $\mathfrak{c}\in Short(-E_8\oplus-\mathbb{Z}^{n-7})$, we have
\begin{equation}\label{eq:12}
\langle\mathfrak{c},\tau\rangle+rp\geq(r-1)p.
\end{equation}
\indent Let $i_{0}$ be the smallest relevant index such that $t_{i_{0}}=0$. Let $i_{1}$ be the smallest relevant index relevant index such that $t_{i_{1}}\leq 1$.

\underline{Case 1: $r$ is odd or $p$ is even}\\
Let $i':=\tfrac{i}{r}$. Lemma~\ref{new 8V} can be rewritten as:

For all $0\leq i'\leq\tfrac{rp}{2}$,
$$8-8t_{ri'}=\max_{\substack{\mathfrak{c}\in Char(Z) \\ \langle\mathfrak{c},\tau\rangle+rp=2i'}}(\mathfrak{c}^2+(n+1)).$$
Hence, when $0\leq i'\leq\tfrac{rp}{2}$,
$$t_{ri'}=0\Leftrightarrow\exists\mathfrak{c}\in short\text{ such that }\langle\mathfrak{c},\tau\rangle+rp=2i'$$
and 
$$t_{ri'}=1\Leftrightarrow t_{ri'}\neq 0 \text{ and }\exists\mathfrak{c}\in Short\text{ such that }\langle\mathfrak{c},\tau\rangle+rp=2i'.$$
By (\ref{eq:12}), we know that
$$i_1=\dfrac{r}{2}\left(rp-C(\tau)\right)\text{ and }i_0=\dfrac{r}{2}\left(rp-c(\tau)\right).$$
\indent Let $j\in PI(c(\tau)+2,C(\tau))$. We want to show that $j\in\{\langle\mathfrak{c},\tau\rangle \mid \mathfrak{c}\in Short(-E_8\oplus-\mathbb{Z}^{n-7})\}$. Let
$$i=\dfrac{r}{2}\left(rp-j\right).$$
Since $j$ has the same parity (odd or even) as $C(\tau)$, $i$ differs from $i_1$ by a multiple of $r$. So, $i$ is relevant. Since $i_1\leq i<i_0$, we must have $t_i=1$. Hence, there exists some $\mathfrak{c}\in Short$ such that
$$\langle\mathfrak{c},\tau\rangle+rp=\dfrac{2i}{r}.$$
This implies $\langle-\mathfrak{c},\tau\rangle=j$.

\underline{Case 2: $r$ is even and $p$ is odd}\\
Let $i':=\tfrac{i}{r}-\tfrac{1}{2}$. Lemma~\ref{new 8V} can be rewritten as:

For all $0\leq i'\leq\tfrac{rp}{2}-1$,
$$8-8t_{ri'}=\max_{\substack{\mathfrak{c}\in Char(Z) \\ \langle\mathfrak{c},\tau\rangle+rp=2i'+1}}(\mathfrak{c}^2+(n+1)).$$
Since $r$ is even, (\ref{eq:12}) implies that for all $\mathfrak{c}\in Short$, we have
$$\langle\mathfrak{c},\tau\rangle+rp-1\geq 0.$$
The rest of the proof of Case 1 applies.
\end{proof}

\begin{Lemma}\label{is E8 changemaker}
$\tau$ is an $E_8$-changemaker.
\end{Lemma}

\begin{proof}
According to the arguments in the proofs of \cite[Prop. 3.6]{caudell} and \cite[Prop. 4.2]{caudell}, $\tau$ is an $E_8$-changemaker when $\langle s,s\rangle=0,-2$. Since $E_8$ is an even lattice, the only other possibility is $\langle s,s\rangle\leq-4$, which Lemma~\ref{-4} says that $\tau$ is an $E_8$-changemaker. Hence, in all cases, $\tau$ is an $E_8$-changemaker.
\end{proof}

Now we can prove Theorem~\ref{poincare knot}.

\begin{proof}[Proof of Theorem~\ref{poincare knot}]
By Lemma~\ref{ortho comp}, Lemma~\ref{is E8 changemaker}, and \cite[Th. 1.19]{caudell}, we conclude that when the intersection form of $Z$ is $-E_8\oplus-\mathbb{Z}^{n-7}$, $L(p,q)$ is a positive integer surgery on a knot in the Poincar\'e homology sphere described by Tange. The rest of the proof of Theorem~\ref{poincare knot} follows from Lemma~\ref{which one is which} and Lemma~\ref{intersection form case 1}.
\end{proof}

\bibliographystyle{amsalpha}
\bibliography{lensbordant}

\end{document}